\documentclass[12pt,a4paper]{article}

\usepackage{latexsym}
\usepackage{amsfonts}
\usepackage{amssymb}
\usepackage{amsmath}
\usepackage{amsthm}
\usepackage{mathrsfs}
\usepackage{enumerate}
\usepackage{slashed}
\usepackage{color}
\usepackage[all]{xy}

\newcommand{\bwedge}{\raisebox{0.2ex}{${\textstyle \bigwedge}$}}

\newcommand{\mathsl}[1]{\mbox{\textsl{\textsf{#1}}$\;\!$}}

\numberwithin{equation}{section}
\newtheorem{dfn}{Definition}
\newtheorem{lem}{Lemma}
\newtheorem{prp}{Proposition}
\newtheorem{thm}{Theorem}
\newtheorem*{te*}{Theorem}

\begin{document}

\title{$C^*$-Completions and the DFR-Algebra}
\author{Michael Forger%
        \thanks{Partly supported by CNPq
                (Conselho Nacional de Desenvolvimento
                 Cient\'{\i}fico e Tecno\-l\'o\-gico), Brazil; \newline
                E-mail: \textsf{forger@ime.usp.br}}
        ~~and~~
        Daniel V.\ Paulino%
        \thanks{Supported by FAPESP
                (Funda\c{c}\~ao de Amparo \`a Pesquisa do
                 Estado de S\~ao Paulo), Brazil; \newline
                 E-mail: \textsf{dberen@ime.usp.br}}
        }
\date{\small{Departamento de Matem\'atica Aplicada \\
      Instituto de Matem\'atica e Estat\'{\i}stica \\
      Universidade de S\~ao Paulo \\
      Caixa Postal 66281 \\
      BR--05314-970~ S\~ao Paulo, S.P., Brazil}}
\maketitle

\thispagestyle{empty}

\begin{abstract}
 \noindent
 The aim of this paper is to present the construction of a general family
 of $C^*$-algebras which includes, as a special case, the ``quantum
 spacetime algebra'' introduced by Doplicher, Fredenhagen and Roberts.
 It is based on an extension of the notion of $C^*$-completion from algebras
 to bundles of algebras, compatible with the usual $C^*$-completion of the 
 appropriate algebras of sections, combined with a novel definition for the
 algebra of the canonical commutation relations using Rieffel's theory of
 strict deformation quantization. Taking the $C^*$-algebra of continuous
 sections vanishing at infinity, we arrive at a functor associating
 a $C^*$-algebra to any Poisson vector bundle and recover the original
 DFR-algebra as a particular example.
\end{abstract}

\begin{flushright}
 \parbox{12em}{
  \begin{center}
   Universidade de S\~ao Paulo \\
   RT-MAP-1201 \\
   Revised version \\
   December 2014
  \end{center}
 }
\end{flushright}

\newpage

\setcounter{page}{1}

\section{Introduction} 

In a seminal paper published in 1995~\cite{DFR}, Doplicher, Fredenhagen and
Roberts (DFR) have introduced a special $C^*$-algebra to provide a model for
spacetime in which localization of events can no longer be performed with
arbitrary precision: they refer to it as a model of ``quantum spacetime''.
Apart from being beautifully motivated, their construction admits a
mathematically simple (re)formulation: it starts from a symplectic
form on Minkowski space and considers the corresponding canonical
commutation relations (CCRs), which can be viewed as a representation
of a well-known finite-dimensional nilpotent Lie algebra, \linebreak
the Heisenberg (Lie) algebra.
More precisely, the CCRs appear in Weyl form, i.e., through an
irreducible, strongly continuous, unitary representation of the
corresponding Heisenberg (Lie) group~-- which, according to the
well-known von Neumann theorem, is unique up to unitary equivalence.
That representation is then used to define a $C^*$-algebra that we 
propose to call the Heisenberg $C^*$-algebra, related to the original
representation through Weyl quantization, that is, via the Weyl-Moyal 
star product.

The main novelty in the DFR construction is that the
underlying symplectic form is treated as a \emph{variable}.
In this way, one is able to reconcile the construction with the principle
of relativistic invariance: since Minkowski space $\mathbb{R}^{1,3}$ has
no distinguished symplectic structure, the only way out is to consider,
simultaneously, \emph{all} possible symplectic structures on Minkowski
space that can be obtained from a fixed one, that is, its orbit~$\Sigma$ 
under the action of the Lorentz group.
This orbit turns out to be isomorphic to $\, TS^2 \times \mathbb{Z}_2$,
thus explaining the origin of the extra dimensions that appear
in this approach.%
\footnote{Note that the factor $\mathbb{Z}_2$ comes from the fact
that we are dealing with the full Lorentz group; it would be absent if
we dropped (separate) invariance under parity $P$ or time reversal $T$.}%
$^,$%
\footnote{The generic feature that any deformation quantization of the function
algebra over Minkowski space must contain, within its classical limit,
some kind of extra factor has been noted and emphasized in~\cite{DVMK}.}

Assuming the symplectic form to vary over the orbit~$\Sigma$ of
some fixed representative produces not just a single Heisenberg
$C^*$-algebra but an entire $C^*$-bundle over this orbit, with the
Heisenberg $C^*$-algebra for the chosen representative as typical fiber.
The continuous sections of that $C^*$-bundle vanishing at infinity then
define a ``section'' $C^*$-algebra, which carries a natural action of
the Lorentz group induced from its natural action on the underlying
bundle of $C^*$-algebras (which moves base points as well as fibers).
Besides, this ``section'' $C^*$-algebra is also a $C^*$-module over the
``scalar'' $C^*$-algebra $C_0(\Sigma)$ of continuous functions on~$\Sigma$
vanishing at infinity.
In the special case considered by DFR, the under\-lying $C^*$-bundle turns
out to be globally trivial, which in view of von Neumann's theorem implies
a classification result on irreducible as well as on Lorentz covariant
representations of the DFR-algebra.

In retrospect, it is clear that when formulated in this geometrically
inspired language, the results of~\cite{DFR} yearn for generalization~--
even if only for purely mathematical reasons.

From a more physical side, one of the original motivations for the present
work was an idea of J.C.A.~Barata, who proposed to look for a clearer
geo\-metrical interpretation of the classical limit of the DFR-algebra
in terms of coherent states, as developed by K.~Hepp~\cite{HE}.
This led the second author to investigate possible generalizations of
the DFR construction to other vector spaces than four-dimensional
Minkowski space and other Lie groups than the Lorentz group in four
dimensions.

As it turned out, the crucial mathematical input for the construction of
the DFR-algebra is a certain symplectic vector bundle over the orbit~$\Sigma$,
namely, the trivial vector bundle $\, \Sigma \times \mathbb{R}^{1,3}$ equipped
with the ``tautological'' symplectic structure, which on the fiber over a
point $\, \sigma \in \Sigma \,$ is just $\sigma$ itself.
Here, we show how, following an approach similar to the one in~%
\cite{Pic}, one can generalize this construction to any Poisson
vector bundle, without supposing homogeneity under some group
action or nondegeneracy of the Poisson tensor.

The basic idea of the procedure is to use the given Poisson structure
to first construct a bundle of Fr\'echet $*$-algebras over the same base
space whose fibers are certain function spaces over the corresponding
fibers of the original vector bundle, the product in each fiber being
the Weyl-Moyal star product given by the Poisson tensor there: the
DFR-algebra is then obtained as the $C^*$-completion of the section
algebra of this Fr\'echet $*$-algebra bundle.
However, a geometrically more appealing interpretation would be as
the section algebra of some $C^*$-bundle, which should be obtained
directly from the underlying Fr\'echet $*$-algebra bundle by a
process of $C^*$-completion.
As it turns out, the concept of a $C^*$-completion at the level of
\emph{bundles} is novel, and one of the main goals we achieve in
this paper is to develop this new theory to the point needed and
then apply it to the situation at hand.

\vspace{1ex}

In somewhat more detail, the paper is organized as follows.

In Section~\ref{sec:*alg}, which is of a preliminary nature, we
gather a few known facts about the construction of $C^*$-completions
of a given $*$-algebra: provided that such completions exist at all,
they can be controlled in terms of the corresponding universal
enveloping $C^*$-algebra, which in particular provides a criterion
for deciding whether such a completion is unique.
Moreover, we notice that when the given $*$-algebra is embedded into
some $C^*$-algebra as a spectrally invariant subalgebra, then that
$C^*$-algebra is in fact its universal enveloping $C^*$-algebra.

In Section~\ref{sec:Halg}, we propose a new definition of ``the $C^*$-%
algebra of the canonical commutation relations'' (for systems with a
finite number of degrees of freedom) which we propose to call the
Heisenberg $C^*$-algebra: it comes in two variants, namely a non%
unital one, $\mathcal{E}_\sigma$, and a unital one, $\mathcal{H}_\sigma$,
obtained as the unique $C^*$-completions of certain Fr\'echet $*$-%
algebras $\mathcal{S}_\sigma$ and $\mathcal{B}_\sigma$, respectively.
These are simply the usual Fr\'echet spaces $\mathcal{S}$ of rapidly
decreasing smooth functions and $\mathcal{B}$ of totally bounded
smooth functions on a given finite-dimensional vector space,%
\footnote{A bounded smooth function on a (finite-dimensional) real
vector space is said to be \emph{totally bounded} if all of its
partial derivatives are bounded.}
equipped with the Weyl-Moyal star product induced by a (possibly
degenerate) bivector~$\sigma$, whose definition, in the unital case,
requires the use of oscillatory integrals as developed in Rieffel's
theory of strict deformation quantization~\cite{RDQ}.
The main advantage of this definition as compared to others that can be found
in the literature~\cite{MSTV,BG1} is that the representation theory of these
$C^*$-algebras corresponds precisely to the representation theory of the
Heisenberg group: as a result of uniqueness of the $C^*$-completion, there
is no need to restrict to a subclass of ``regular'' representations.

In Section~\ref{sec:*bun}, we present the core material of this paper.
We begin by introducing the concept of a bundle of locally convex $*$-%
algebras, which contains that of a $C^*$-bundle as a special case, and
following the approach of Dixmier, Fell and other authors~\cite{Dix,FD},
we show how the topology of the total space of any such bundle is tied
to its algebra of continuous sections.
Next, we pass to the $C^*$ setting, where we explore the notion of a
$C_0(X)$-algebra ($X$ being some fixed locally compact topological space).
At first sight, this appears to generalize the natural module structure
of the section algebra of a $C^*$-bundle over~$X$, but according to the
sectional representation theorem~\cite[Theorem~C.26, p.~367]{Wil}, it
actually provides a necessary and sufficient condition for a $C^*$-%
algebra to be the section algebra of a $C^*$-bundle over~$X$.
Here, we formulate a somewhat strengthened version of that theorem
which establishes a categorical equivalence between $C^*$-bundles
over~$X$ and $C_0(X)$-algebras.
Finally, we introduce the (apparently novel) concept of $C^*$-completion
of a bundle of locally convex $*$-algebras and show that, using this
(essentially fiberwise) definition and imposing appropriate conditions
on the behavior of sections at infinity, the two processes of completion
and of passing to section algebras commute: the $C^*$-completion of
the algebra of continuous sections with compact support of a bundle
of locally convex $*$-algebras is naturally isomorphic to the algebra
of continuous sections vanishing at infinity of its $C^*$-completion.

In Section~\ref{sec:DFRPois}, we combine the methods developed in the
previous two sections to construct, from an arbitrary Poisson vector
bundle $E$ over an arbitrary manifold~$X$, with Poisson tensor~$\sigma$,
two bundles of Fr\'echet $*$-algebras over~$X$, $\mathcal{S}(E,\sigma)$
and $\mathcal{B}(E,\sigma)$, as well as two $C^*$-bundles over~$X$,
$\mathcal{E}(E,\sigma)$ and $\mathcal{H}(E,\sigma)$, the latter being
the $C^*$-completions of the former with respect to the $C^*$ fiber
norms induced by the unique $C^*$-norms on each fiber, according to
the prescriptions of Section~\ref{sec:Halg}.
We propose to refer to these $C^*$-bundles as DFR-bundles and to the
corresponding section algebras as DFR-algebras, since we show that
the original DFR-algebra can be recovered as a special case, by an
appropriate and natural choice of Poisson vector bundle.
Moreover, that construction can be applied fiberwise to the tangent
spaces of any Lorentzian manifold to define a functor from the category
of Lorentzian manifolds (of fixed dimension) to that of $C^*$-algebras
which might serve as a starting point for a notion of ``locally covariant
quantum spacetime''.

\vspace{1ex}

The overall picture that emerges is that the constructions presented
in this paper establish a systematic method for producing a vast class
of examples of $C^*$-algebras provided with additional ingredients
that are tied up with structures from classical differential geo\-metry
and/or topology in a functorial manner.
To what extent this new class of examples can be put to good use
remains to be seen.
But we believe that even the original question of how to define the
classical limit of the DFR-algebra, or more generally how to handle
its space of states, will be deeply influenced by the generalization
presented here, which is of independent mathematical interest, going
way beyond the physical motivations of the original DFR paper.

\pagebreak

\section{\emph{C}$^*$-Completions of $*$-Algebras}
\label{sec:*alg}

As a preliminary step that will be needed for the constructions to
be presented later on, we want to discuss the question of existence
and uniqueness of the $C^*$-completion of a $*$-algebra (possibly
equipped with some appropriate locally convex topology of its own),
which is closely related to the concept of a spectrally invariant
subalgebra, as well as the issue of continuity of the inversion
map on the group of invertible elements.

We begin by recalling a general and well-known strategy for producing
$C^*$-norms on $*$-algebras. It starts from the observation that given
any $*$-algebra~$B$ and any $*$-representation $\rho$ of~$B$ on a
Hilbert space~$\mathfrak{H}_\rho$, we can define a $C^*$-seminorm
$\|.\|_\rho$ on~$B$ by taking the operator norm in~$B(\mathfrak{H}_\rho)$,
i.e., by setting, for any $\, b \in B$,
\begin{equation} \label{eq:OPSND} 
 \|b\|_\rho^{}~=~\|\rho(b)\|~.
\end{equation}
Obviously, this will be a $C^*$-norm if and only if $\rho$ is faithful.
More generally, given any set $R$ of $*$-representations of~$B$ such
that, for any $\, b \in B$, $\{ \|\rho(b)\| \, | \, \rho \in R \} \,$
is a bounded subset of $\mathbb{R}$, setting
\begin{equation} \label{eq:MAXCSTN} 
 \|b\|_R^{}~=~\sup_{\rho \in R} \|b\|_\rho^{}
\end{equation}
will define a $C^*$-seminorm on~$B$, which is even a $C^*$-norm
as soon as the set $R$ separates~$B$ (i.e., for any $\, b \in B
\setminus \{0\}$, there exists $\, \rho \in R \,$ such that
$\, \rho(b) \neq 0$).
Taking into account that every $C^*$-seminorm $s$ on $B$ is the operator
norm for some $*$-representation of~$B$ (this follows from applying the
Gelfand-Naimark theorem~\cite[Theorem 3.4.1]{MU} to the $C^*$-completion
of~$B/\ker s$, together with the fact that every faithful $C^*$-algebra
representation is automatically isometric~\cite[Theorem 3.1.5]{MU}), we
can take $R$ to be the set $\text{Rep}(B)$ of all $*$-representations
of~$B$ (up to equivalence) to obtain a $C^*$-seminorm on~$B$ which is
larger than any other one, provided that, for any $\, b \in B$,
\begin{equation} \label{eq:OPSNB} 
 \sup_{\rho \in \text{Rep}(B)} \|b\|_\rho^{}~<~\infty~.
\end{equation}
Moreover, when $\text{Rep}(B)$ separates~$B$, we obtain the well-known
\emph{maximal} $C^*$-norm on~$B$, which gives rise to the \emph{minimal
$C^*$-completion} of~$B$, also denoted by $C^*(B)$ and called the \emph%
{universal enveloping $C^*$-algebra} of~$B$ because it satisfies the
following universal property: for every $C^*$-algebra $C$, every
$*$-algebra homomorphism from~$B$ to~$C$ extends uniquely to a
$C^*$-algebra homomorphism from $C^*(B)$ to~$C$.

Next, given a $*$-algebra $B$ embedded in some $C^*$-algebra $A$ as
a dense $*$-subalgebra, one method for guaranteeing existence of the
universal enveloping $C^*$-algebra relies on the concept of spectral
invariance, which is defined as follows:  $B$ is said to be \emph%
{spectrally invariant} in~$A$ if, for every element of~$B$, its
spectrum in~$B$ is the same as its spectrum in~$A$.
(Note that, in general, the former contains the latter:
thus only the opposite inclusion is a nontrivial condition.%
\footnote{Here, we use that if $A$ and $B$ are $*$-algebras and
$\, \phi: B \to A \,$ is any $*$-algebra homomorphism, then for
any $b$ in~$B$, the spectrum $\sigma_B(b)$ of~$b$ in~$B$ contains
the spectrum $\sigma_A(\phi(b))$ of~$\phi(b)$ in~$A$ and hence we
have $\, r_A(\phi(b)) \leqslant r_B(b)$: this follows directly
from the definition of the spectrum of elements of $*$-algebras.
})
In this situation, we may conclude that, for any self-adjoint
element $b$ of~$B$,
\[
 \sup_{\rho \in \text{Rep}(B)} \|b\|_\rho^{}~\leqslant~r(b)~,
\]
where $r(b)$ denotes the spectral radius of~$b$ in~$B$, which
by hypothesis coincides with its spectral radius in $A$ and
hence (for self-adjoint~$b$) also with its $C^*$-norm in~$A$.
But this means that the $C^*$-norm in~$A$ is in fact the
maximal $C^*$-norm and hence that the $C^*$-algebra $A$
is precisely the universal enveloping $C^*$-algebra of~$B$:
$A = C^*(B)$.

As an example showing the usefulness of this concept, we note the following
\begin{thm} \label{thm:APID}
 Let\/ $A$ be a (nonunital) $C^*$-algebra, equipped with the standard
 partial \linebreak ordering induced by the cone\/~$A^+$ of positive
 elements, and let\/~$B$ be a spectrally invariant $*$-subalgebra of\/~$A$.
 Then\/ $A$ admits an approximate identity consisting of elements of\/~$B$,
 i.e., a directed set\/ $(e_\lambda)_{\lambda \in \Lambda}$ of elements\/
 $e_\lambda$ of\/~$B$ such that, in\/~$A$, $e_\lambda \geqslant 0$,
 $\|e_\lambda\| \leqslant 1$, $e_\lambda \leqslant e_\mu$ \linebreak
 if $\, \lambda \leqslant \mu \,$ and, for every $\, a \in A$,
 $\lim_\lambda \, e_\lambda a = a = \lim_\lambda \, a e_\lambda$.
\end{thm}

\noindent
The proof is an easy adaptation of that of a similar theorem due to
Inoue, in the context of locally $C^*$-algebras, for which we refer
the reader to~\cite[Theorem~11.5]{Fra}; we note here that the version
given above can also be generalized to locally $C^*$-algebras without
additional effort. The main difference is that we assume $B$ to be
just a dense $*$-subalgebra, rather than a dense $*$-ideal, and
spectral invariance turns out to be the crucial ingredient to
make the proof work.

Once the existence of the universal enveloping $C^*$-algebra $C^*(B)$
of~$B$ is settled~-- usually by realizing it explicitly as a spectrally
invariant $*$-subalgebra of a given $C^*$-algebra~$A$~-- we can address
the question of classifying \emph{all} possible $C^*$-norms on~$B$,
which correspond to the possible $C^*$-seminorms on~$A$ and are
characterized by their respective kernels: these are ideals in~$A$
that intersect $B$ trivially.%
\footnote{In this paper, where we deal exclusively with topological
$*$-algebras, ``ideal'' will mean ``closed two-sided $*$-ideal'',
unless explicitly stated otherwise.\label{fn:IDEAL}}
Thus if we can determine what are the ideals in~$A$ and prove that none
of them intersects $B$ trivially, we can conclude that~$B$ admits one
and only one $C^*$-norm.

Finally, it is worth noting that in many cases of interest, $B$ will not
be merely a $*$-algebra but will come equipped with a (locally convex)
topology of its own, with respect to which it is complete.
Within this context, we have the following result which will
become useful later on:
\begin{prp} \label{prp:INVER}
 Let~$B$ be a Fr\'echet $*$-algebra, i.e., a $*$-algebra which is also a
 Fr\'echet space such that multiplication and involution are continuous,
 and assume that $B$ is continuously embedded in some $C^*$-algebra~$A$
 as a spectrally invariant $*$-subalgebra. Then the group~$G_B$ of
 invertible elements of~$B$ is open and the inversion map
 \[
  \begin{array}{ccc}
   G_B & \longrightarrow &  G_B  \\[1mm]
    b  &   \longmapsto   & b^{-1}
  \end{array}
 \]
 is continuous not only in the induced $C^*$-topology but also in the
 Fr\'echet topology of~$B$.
\end{prp}
\begin{proof}
 The statement of this proposition is well-known for the $C^*$-topology,
 but that it also holds for the finer Fr\'echet topology is far from obvious,
 as can be inferred from the extensive discussion of concepts related to
 this question that can be found in the literature, such as that of
 ``Q-algebras'' and of ``topological algebras with inverses''; see
 \cite[Chapter~1, Section~6]{Fra}, \cite[Chapter~3, Section~6]{Bal}
 and references therein.
 In the present context, spectral invariance guarantees that $G_B$ is
 equal to $B \cap G_A$, i.e., it is the inverse image of $G_A$, which
 is open in~$A$, under the inclusion map $\, B \hookrightarrow A$,
 which by hypothesis is continuous. Continuity of the inversion map
 then follows from the Arens-Banach theorem~\cite[Theorem~3.6.16]{Bal}
 or from a more general direct argument~\cite{Pfi}.
\end{proof}

\section{The Heisenberg \emph{C}$^*$-Algebra for \\
         Poisson Vector Spaces}
\label{sec:Halg}

Let $V$ be a Poisson vector space, i.e., a real vector space of
dimension~$n$, say, equipped with a fixed bivector $\sigma$ of
rank $2r$; in other words, the dual $V^*$ of $V$ is a
presymplectic vector space.%
\footnote{Note that we do \emph{not} require $\sigma$ to be
nondegenerate.}
It gives rise to an $(n+1)$-dimensional Lie algebra
$\mathfrak{h}_{\sigma}$ which is a one-dimensional
central extension of the abelian Lie algebra~$V^*$
defined by the cocycle~$\sigma$ and will be called
the \emph{Heisenberg algebra} or, more precisely,
\emph{Heisenberg Lie algebra} (associated to $V^*$
and~$\sigma$): as a vector space, $\mathfrak{h}_\sigma
= V^* \oplus \mathbb{R}$, with commutator given by
\begin{equation} \label{eq:HACOMM}
 [(\xi,\lambda),(\eta,\mu)]~=~(0 \,,\, \sigma(\xi,\eta))
 \qquad \mathrm{for}~~\xi,\eta \in V^* \,,\,
 \lambda,\mu \in \mathbb{R}~.
\end{equation}
Associated with this Lie algebra is the \emph{Heisenberg group}
or, more precisely, \emph{Heisenberg Lie group}, $H_\sigma$: as
a manifold, $H_\sigma = V^* \times \mathbb{R}$, with product
given by
\begin{equation} \label{eq:HGPROD}
 (\xi,\lambda) \, (\eta,\mu)~=~
 \bigl( \xi+\eta \,,\,
        \lambda + \mu - {\textstyle \frac{1}{2}} \, \sigma(\xi,\eta) \bigr)
 \qquad \mathrm{for}~~\xi,\eta \in V^\ast \,,\,
                      \lambda,\mu \in \mathbb{R}~.
\end{equation}

In what follows, we shall discuss various forms of giving a precise
mathematical meaning to the concept of a \emph{representation of the
canonical commutation relations} defined by~$\sigma$.
From the very beginning, we shall restrict ourselves to representations
that can be brought into \emph{Weyl form}, i.e., that correspond to
strongly continuous unitary representations $\pi$ of the Heisenberg
group $H_\sigma$: abbreviating $\pi(\xi,0)$ to $\pi(\xi)$, these
relations can be written in the form
\begin{equation} \label{eq:CCR1}
 \pi(\xi) \, \pi(\eta)~
 =~e^{-\frac{i}{2} \sigma(\xi,\eta)} \, \pi(\xi+\eta)~.
\end{equation}
At the infinitesimal level, they correspond to representations $\dot{\pi}$
of the Heisenberg algebra~$\mathfrak{h}_\sigma$ which are often called
``regular'': according to Nelson's theorem~\cite{NE}, these are precisely
the representations of~$\mathfrak{h}_\sigma$ by essentially skew adjoint
operators on a common dense invariant domain of analytic vectors.

Our main goal in this section is to use these representations of the
canonical commutation relations to construct what we shall call the
\emph{Heisenberg $C^*$-algebra}: more precisely, this algebra comes
in two versions, namely a nonunital one and a unital one, denoted
here by $\mathcal{E}_\sigma$ and by $\mathcal{H}_\sigma$, respectively:
as it turns out, the latter is simply the multiplier algebra of the
former.
We emphasize that our construction differs substantially from
previous ones that can be found in the literature, such as the
Weyl algebra of Refs~\cite{MA,MSTV} or the resolvent algebra of
Ref.~\cite{BG1}: both of those use the method of constructing a
$C^*$-algebra from an appropriate set of generators and relations.
Instead, we focus on certain Fr\'echet $*$-algebras that play a central
role in Rieffel's theory of strict deformation quantization~\cite{RDQ}
and show that each of these admits a unique $C^*$-norm, so it has a
unique $C^*$-completion.
For further comments, we refer the reader to the end of this section.



\subsection{The Heisenberg-Schwartz and Heisenberg-Rieffel Algebras}

Throughout this paper, given any (finite-dimensional) real vector space~$W$,
say, we denote by $\mathcal{S}(W)$ the Schwartz space of rapidly decreasing
smooth functions on~$W$ and by $\mathcal{B}(W)$ the space of totally bounded
smooth functions on~$W$.\footnotemark[2]

To begin with, we want to briefly recall how one can use the bivector~%
$\sigma$ to introduce a new product on the space $\mathcal{S}(V)$ which
is a deformation of the standard pointwise product, commonly known as
the \emph{Weyl-Moyal star product}, and will then comment on how that
deformed product can be extended to the space~$\mathcal{B}(V)$.

We start by noting that given any strongly continuous unitary
representation $\pi$ of the Heisenberg group~$H_\sigma$ on some
Hilbert space~$\mathfrak{H}_\pi$, we can construct a continuous
linear map
\begin{equation} \label{eq:WQUANT1}
 \begin{array}{cccc}
  W_\pi: & \mathcal{S}(V) & \longrightarrow & B(\mathfrak{H}_\pi) \\[1mm]
         &       f        &   \longmapsto   &      W_\pi f
 \end{array}
\end{equation}
from $\mathcal{S}(V)$ to the space of bounded linear operators on~%
$\mathfrak{H}_\pi$, called the \emph{Weyl quantization map}, by setting
\begin{equation} \label{eq:WQUANT2}
 W_\pi f~=~\int_{V^*} d\xi~\check{f}(\xi) \; \pi(\xi)~,
\end{equation}
which is to be compared with
\begin{equation} \label{eq:FOURT1}
 f(x)~=~\int_{V^*} d\xi~\check{f}(\xi) \; e^{i\langle\xi,x\rangle}~,
\end{equation}
where $\check{f}$ is the inverse Fourier transform of~$f$,
\begin{equation} \label{eq:FOURT2}
 \check{f}(\xi)~\equiv~(\mathcal{F}^{-1} f)(\xi)~
                =~\frac{1}{(2\pi)^n} \;
                  \int_V dx~f(x) \; e^{-i\langle\xi,x\rangle}~.
\end{equation}
Note that equation~(\ref{eq:WQUANT2}) should be understood as stating
that, for every vector $\psi$ in $\mathfrak{H}_\pi$, we have
\[
 (W_\pi f) \psi~=~\int_{V^*} d\xi~\check{f}(\xi) \; \pi(\xi) \psi~,
\]
since it is this integral that makes sense as soon as $\pi$ is strongly
continuous; then it is obvious that $\, W_\pi f \in B(\mathfrak{H}_\pi)$,
with $\, \| W_\pi f \| \leqslant \| \check{f} \|_1^{}$, where $\| . \|_1^{}$
is the $L^1$-norm on~$\mathcal{S}(V^*)$ which, as shown in the appendix
(see equation~(\ref{eq:L1EST})), can be estimated in terms of a suitable
Schwartz seminorm of~$f$:
\begin{equation} \label{eq:SCHWEST1}
 \| W_\pi f \|~\leqslant~\| \check{f} \|_1^{}~\leqslant~
 (2\pi)^n \; \sum_{|\alpha|,|\beta| \leqslant 2n} \, \sup_{x \in V} \,
 |x^\alpha \, \partial_\beta^{} f (x)|~.
\end{equation}
Moreover, an explicit calculation shows that, independently of the choice
of~$\pi$, we have, for $\, f,g \in \mathcal{S}(V)$,
\begin{equation} \label{eq:WMOY1}
 W_\pi f~W_\pi g~=~W_\pi (f \star_\sigma g)~,
\end{equation}
where $\star_\sigma$ denotes the \emph{Weyl-Moyal star product} of
$f$ and~$g$, which is given by any one of the following two twisted
convolution integrals:
\begin{equation} \label{eq:WMOY2}
 (f \star_\sigma g)(x)~=~\int_{V^*} d\xi~e^{i\langle\xi,x\rangle}
 \int_{V^*} d\eta~\check{f}(\eta) \, \check{g}(\xi-\eta) \;
 e^{\frac{i}{2} \sigma(\xi,\eta)}~.
\end{equation}
\begin{equation} \label{eq:WMOY3}
 (f \star_\sigma g)(x)~=~\int_{V^*} d\xi~e^{i\langle\xi,x\rangle}
 \int_{V^*} d\eta~\check{f}(\xi-\eta) \, \check{g}(\eta) \;
 e^{-\frac{i}{2} \sigma(\xi,\eta)}~.
\end{equation}
The proof is a simple computation (we omit the $\psi$):
\begin{eqnarray*}
 W_\pi f~W_\pi g \!\!
 &=&\!\! \int_{V^*} d\eta \int_{V^*} d\zeta~\check{f}(\eta) \, \check{g}(\zeta)~
         \pi(\eta) \, \pi(\zeta) \\[2mm]
 &=&\!\! \int_{V^*} d\eta \int_{V^*} d\zeta~\check{f}(\eta) \, \check{g}(\zeta)
         \; e^{-\frac{i}{2} \sigma(\eta,\zeta)} \, \pi(\eta+\zeta) \\[2mm]
 &=&\!\! \int_{V^*} d\eta \int_{V^*} d\xi~\check{f}(\eta) \,
         \check{g}(\xi-\eta) \; e^{-\frac{i}{2} \sigma(\eta,\xi)} \, \pi(\xi)
 \\[2mm]
 &=&\!\! \int_{V^*} d\xi \int_{V^*} d\eta~\check{f}(\eta) \,
         \check{g}(\xi-\eta) \; e^{\frac{i}{2} \sigma(\xi,\eta)} \, \pi(\xi)
 \\[2mm]
 &=&\!\! \int_{V^*} d\xi~\mathcal{F}^{-1}(f \star_\sigma g)(\xi) \, \pi(\xi)
 \\[3mm]
 &=&\!\! W_\pi (f \star_\sigma g)~.
\end{eqnarray*}

\noindent
For the sake of comparison, we note an alternative form of
this product~\cite{RDQ}: using the ``musical homomorphism''
$\, \sigma^\sharp: V^* \longrightarrow V \,$ induced by
$\sigma$ (i.e., $\langle \xi , \sigma^\sharp \eta \rangle
= \sigma(\eta,\xi)$), we get

\pagebreak

\begin{eqnarray*}
\lefteqn{(f \star_\sigma g)(x)~
         =~\int_{V^*} d\xi~e^{i\langle\xi,x\rangle} \int_{V^*} d\eta~
           \check{f}(\eta) \, \check{g}(\xi-\eta) \;
           e^{\frac{i}{2} \sigma(\xi,\eta)}} \hspace*{5mm} \\[2mm]
 &=& \frac{1}{(2\pi)^{2n}} \,
     \int_{V^*} d\xi~e^{i\langle\xi,x\rangle} \, \int_{V^*} d\eta \,
     \int_V dv~f(v) \; e^{-i\langle \eta,v \rangle} \,
     \int_V dw~g(w) \; e^{-i\langle \xi-\eta,w \rangle} \,
     e^{-\frac{i}{2} \langle\xi,\sigma^\sharp \eta\rangle} \\[2mm]
 &=& \frac{1}{(2\pi)^n} \,
     \int_{V^*} d\eta \, \int_{V} dv~f(v) \, \int_{V} dw~g(w)
     \left( \int_{V^*} \frac{d\xi}{(2\pi)^n}~
            e^{-i\langle \xi , w-(x-\frac{1}{2} \sigma^\sharp \eta) \rangle}
            \right) e^{i\langle \eta,w-v \rangle} \\[2mm]
 &=& \frac{1}{(2\pi)^n} \, \int_{V^*} d\eta \, \int_V dv~
     f(v) \, g(x-{\textstyle \frac{1}{2}} \sigma^\sharp \eta) \,
     e^{i\langle \eta,x-\frac{1}{2} \sigma^\sharp \eta-v \rangle}~,
\end{eqnarray*}
and similarly
\begin{eqnarray*}
\lefteqn{(f \star_\sigma g)(x)~
         =~\int_{V^*} d\xi~e^{i\langle\xi,x\rangle} \int_{V^*} d\eta~
           \check{f}(\xi-\eta) \, \check{g}(\eta) \;
           e^{-\frac{i}{2} \sigma(\xi,\eta)}} \hspace*{5mm} \\[2mm]
 &=& \frac{1}{(2\pi)^{2n}} \,
     \int_{V^*} d\xi~e^{i\langle\xi,x\rangle} \, \int_{V^*} d\eta \,
     \int_V dv~f(v) \; e^{-i\langle \xi-\eta,v \rangle} \,
     \int_V dw~g(w) \; e^{-i\langle \eta,w \rangle} \,
     e^{\frac{i}{2} \langle\xi,\sigma^\sharp \eta\rangle} \\[2mm]
 &=& \frac{1}{(2\pi)^n} \,
     \int_{V^*} d\eta \, \int_{V} dv~f(v) \, \int_{V} dw~g(w)
     \left( \int_{V^*} \frac{d\xi}{(2\pi)^n}~
            e^{-i\langle \xi , v-(x+\frac{1}{2} \sigma^\sharp \eta) \rangle}
            \right) e^{i\langle \eta,v-w \rangle} \\[2mm]
 &=& \frac{1}{(2\pi)^n} \, \int_{V^*} d\eta \, \int_V dw~
     f(x+{\textstyle \frac{1}{2}} \sigma^\sharp \eta) \, g(w) \,
     e^{i\langle \eta,x+\frac{1}{2} \sigma^\sharp \eta-w \rangle}~,
\end{eqnarray*}
i.e., after a change of variables $\, v \rightarrow u = v-x$,
$\, \eta \rightarrow \xi = \eta/2\pi \,$ in the first case and
$\, w \rightarrow v = w-x$, $\, \eta \rightarrow \xi = - \eta/2
\pi \,$ in the second case,
\begin{equation} \label{eq:WMOY4}
 (f \star_\sigma g)(x)~
 =~\int_{V^*} d\xi \, \int_V du~f(x+u) \, g(x - \pi \sigma^\sharp \xi) \;
   e^{-2\pi i \langle \xi, u \rangle}~.
\end{equation}
\begin{equation} \label{eq:WMOY5}
 (f \star_\sigma g)(x)~
 =~\int_{V^*} d\xi \, \int_V dv~f(x - \pi \sigma^\sharp \xi) \, g(x+v) \;
   e^{2\pi i \langle \xi,v \rangle}~.
\end{equation}
At any rate, with respect to the Weyl-Moyal star product, together
with the standard involution of pointwise complex conjugation and
the standard Fr\'echet topology, the space $\mathcal{S}(V)$ becomes
a Fr\'echet $*$-algebra, which we shall denote by $\mathcal{S}_\sigma$
and call the \emph{Heisenberg-Schwartz algebra} (with respect to $\sigma$).%
\footnote{Continuity of the Weyl-Moyal star product with respect to
the standard Fr\'echet topology on~$\mathcal{S}(V)$ is well-known
and is an immediate consequence of Proposition~\ref{prp:SCHWEST}
in the appendix.}

Dealing with the Weyl-Moyal star product between two functions in
$\mathcal{B}(V)$, rather than $\mathcal{S}(V)$, is substantially
more complicated.
In this case, its definition is based on equation~(\ref{eq:WMOY4})
or equation~(\ref{eq:WMOY5}), whose rhs has to be interpreted as an
oscillatory integral on~$V^* \times V$.
Fortunately, all of the necessary analytic tools have been provided by
Rieffel~\cite{RDQ} (with the identification $\, J = - \pi \sigma^\sharp$),
so we may just state, as one of the results, that with respect to the
Weyl-Moyal star product, together with the standard involution of point%
wise complex conjugation and the standard Fr\'echet topology, the space
$\mathcal{B}(V)$ becomes a Fr\'echet $*$-algebra, which we shall denote
by $\mathcal{B}_\sigma$ and propose to call the \emph{Heisenberg-Rieffel
algebra} (with respect to $\sigma$).

We note in passing that both algebras are noncommutative when $\, \sigma
\neq 0$, but their deviation from commutativity is explicitly controlled
by a simple formula:
\begin{equation} \label{eq:NONCOM}
 g \star_\sigma f~=~f \star_{-\sigma} g~.
\end{equation}

Returning to explicit integral formulas, we note next that an intermediate
situation, which will be of particular interest in what follows, occurs
when one factor belongs to $\mathcal{B}(V)$ while the other belongs to
$\mathcal{S}(V)$, since we may perform a change of variables \linebreak
$\, u \rightarrow y = u+x$, $\, \xi \rightarrow \eta = 2\pi \xi \,$
(and rename $\eta$ to $\xi$) to rewrite equation~(\ref{eq:WMOY4})
in the form
\begin{equation} \label{eq:WMOY6}
 (f \star_\sigma g)(x)~
 =~\int_{V^*} d\xi~\check{f}(\xi) \,
   g(x - {\textstyle \frac{1}{2}} \sigma^\sharp \xi) \;
   e^{i \langle \xi, x \rangle}~,
\end{equation}
and similarly, we may perform a change of variables $v \rightarrow y = x+v$,
$\, \xi \rightarrow \eta = - 2\pi \xi \,$ (and rename $\eta$ to $\xi$) to
rewrite equation~(\ref{eq:WMOY5}) in the form
\begin{equation} \label{eq:WMOY7}
 (f \star_\sigma g)(x)~
 =~\int_{V^*} d\xi~f(x + {\textstyle \frac{1}{2}} \sigma^\sharp \xi) \,
   \check{g}(\xi) \; e^{i \langle \xi,x \rangle}~.
\end{equation}
Note that the expression in equation~(\ref{eq:WMOY6}) makes sense
when $\, f \in \mathcal{S}(V)$, $g \in \mathcal{B}(V)$ and similarly, 
the expression in equation~(\ref{eq:WMOY7}) makes sense when $\, f
\in \mathcal{B}(V)$, $g \in \mathcal{S}(V)$: both are then ordinary
integrals that become iterated integrals when the expression~%
(\ref{eq:FOURT2}) for the inverse Fourier transform is written
out explicitly.
(Obviously, the two formulae can be converted into each other by
means of equation~(\ref{eq:NONCOM}).)
Moreover, it follows from elementary estimates which can be found
in the appendix (see Proposition~\ref{prp:SCHWEST}) that, in either
case, $f \star_\sigma g \in \mathcal{S}(V)$, and the linear operators
\begin{equation} \label{eq:LREGR0}
 \begin{array}{cccc}
  L_\sigma f: & \mathcal{S}_\sigma & \longrightarrow & \mathcal{S}_\sigma
  \\[1mm]
              &        h           &   \longmapsto   &  f \star_\sigma h
 \end{array}
\end{equation}
of left translation by $\, f \in \mathcal{B}_\sigma \,$ and
\begin{equation} \label{eq:RREGR0}
 \begin{array}{cccc}
  R_\sigma g: & \mathcal{S}_\sigma & \longrightarrow & \mathcal{S}_\sigma
  \\[1mm]
              &        h           &   \longmapsto   &  h \star_\sigma g
 \end{array}
\end{equation}
of right translation by $\, g \in \mathcal{B}_\sigma \,$ are continuous
in the Schwartz topology.
In particular, $\mathcal{S}_\sigma$ is a two-sided $*$-ideal in~%
$\mathcal{B}_\sigma$ (but note that $\mathcal{S}_\sigma$ is not closed
in~$\mathcal{B}_\sigma$); see~\cite[Chapter~3]{RDQ} for more details.
Thus we get a $*$-homomorphism
\begin{equation} \label{eq:MULTA1}
 \begin{array}{ccc}
  \mathcal{B}_\sigma & \longrightarrow &  M(\mathcal{S}_\sigma) \\[1mm]
          f         &   \longmapsto   & (L_\sigma f,R_\sigma f) 
 \end{array}
\end{equation}
which provides an embedding of $\mathcal{B}_\sigma$ into what might be
called the multiplier algebra $M(\mathcal{S}_\sigma)$ of~$\mathcal{S}_\sigma$.
However, we have refrained from using this terminology since there is no
established definition of the concept of multiplier algebra beyond the
realm of Banach algebras: there are ``a priori'' many possible candidates
for its locally convex topology.%
\footnote{This generic statement does not exclude the existence of
special cases where the ``most natural'' ones among these topologies
coincide, as happens in the case of~$M(\mathcal{S}_\sigma)$ when
$\sigma$ is nondegenerate~\cite{VGB}.}
Of course, this ambiguity will no longer be a problem as soon as we pass
to the $C^*$-completions.

\subsection{Existence of $C^*$-norms}

The Fr\'echet algebras $\mathcal{S}_\sigma$ and $\mathcal{B}_\sigma$
both admit various norms. The naive choice would be the standard
sup norm, but this is a $C^*$-norm for the usual pointwise product,
not for the Weyl-Moyal star product. Hence the first question is
whether there exist $C^*$-norms on~$\mathcal{S}_\sigma$ and on
$\mathcal{B}_\sigma$ at all.
Fortunately, the answer is affirmative: it suffices to take the
operator norm in the regular representation.
More precisely, consider the $*$-representation
\begin{equation} \label{eq:LREGR1}
 \begin{array}{cccc}
  L_\sigma: & \mathcal{S}_\sigma & \longrightarrow & B(L^2(V)) \\[1mm]
            &         f          &   \longmapsto   & L_\sigma f
 \end{array}
\end{equation}
of~$\mathcal{S}_\sigma$, which extends to a $*$-representation
\begin{equation} \label{eq:LREGR2}
 \begin{array}{cccc}
  L_\sigma: & \mathcal{B}_\sigma & \longrightarrow & B(L^2(V)) \\[1mm]
            &         f          &   \longmapsto   & L_\sigma f
 \end{array}
\end{equation}
of~$\mathcal{B}_\sigma$, both defined by taking the operator $\, L_\sigma f:
L^2(V) \longrightarrow L^2(V) \,$ to be the unique continuous linear
extension of the operator $\, L_\sigma f: \mathcal{S}(V) \longrightarrow
\mathcal{S}(V) \,$ of equation~(\ref{eq:LREGR0}).
Similarly, we may also consider the (anti-)$*$-representation
\begin{equation} \label{eq:RREGR1}
 \begin{array}{cccc}
  R_\sigma: & \mathcal{S}_\sigma & \longrightarrow & B(L^2(V))  \\[1mm]
            &         g          &   \longmapsto   & R_\sigma g
 \end{array}
\end{equation}
of~$\mathcal{S}_\sigma$, which extends to an (anti-)$*$-representation
\begin{equation} \label{eq:RREGR2}
 \begin{array}{cccc}
  R_\sigma: & \mathcal{B}_\sigma & \longrightarrow & B(L^2(V))  \\[1mm]
            &         g          &   \longmapsto   & R_\sigma g
 \end{array}
\end{equation}
of~$\mathcal{B}_\sigma$, both defined by taking the operator $\, R_\sigma g:
L^2(V) \longrightarrow L^2(V) \,$ to be the unique continuous linear
extension of the operator $\, R_\sigma g: \mathcal{S}(V) \longrightarrow
\mathcal{S}(V) \,$ of equation~(\ref{eq:RREGR0}). \linebreak
Obviously, any $L_\sigma f$ commutes with any $R_\sigma g$: this
is nothing but associativity of the star product.
Of~course, this construction presupposes that the operators $L_\sigma f$
of equation~(\ref{eq:LREGR0}) (and analogously, the operators $R_\sigma g$
of equation~(\ref{eq:RREGR0})) are continuous not only in the Schwartz
topology but also in the $L^2$-norm.
Moreover, we need to show that the linear maps $L_\sigma$ in
equations~(\ref{eq:LREGR1}) and~(\ref{eq:LREGR2}) (and analogously, the
linear maps $R_\sigma$ in equations~(\ref{eq:RREGR1}) and~(\ref{eq:RREGR2}))
are continuous with respect to the appropriate topologies.
And finally, we want these continuity properties to hold
locally uniformly when we vary~$\sigma$.
Fortunately, all these statements can be derived from a single estimate,
as we explain in what follows.

First, consider the case when $f$ belongs to~$\mathcal{S}(V)$: then we can
rewrite equation~(\ref{eq:WMOY6}) in the form of equation~(\ref{eq:WQUANT2}),
since
\begin{equation}
 L_\sigma f~=~\int_{V^*} d\xi~\check{f}(\xi) \; \pi^{\mathrm{reg}}(\xi)~,
\end{equation}
where $\pi^{\mathrm{reg}}$ is the regular representation of the Heisenberg group
$H_\sigma$, that is, the strongly continuous unitary representation of~$H_\sigma$
on the Hilbert space~$L^2(V)$ defined by setting
\begin{equation}
 (\pi^{\mathrm{reg}}(\xi) \psi)(x)~
 =~e^{i \langle \xi, x \rangle} \,
   \psi(x - {\textstyle \frac{1}{2}} \sigma^\sharp \xi)~,
\end{equation}
i.e., $\pi^{\mathrm{reg}}(\xi)$ is the operator of translation by~%
$-\frac{1}{2} \sigma^\sharp \xi$ followed by that of multiplication
with the phase function $e^{i \langle \xi, \,.\, \rangle}$.
As before, it follows that $\, L_\sigma f \in B(L^2(V))$, with
$\, \| L_\sigma f \| \leqslant \| \check{f} \|_1^{}$, \linebreak where
$\| . \|_1^{}$ is the $L^1$-norm on~$\mathcal{S}(V^*)$ which, as shown in
the appendix (see equation~(\ref{eq:L1EST})), can be estimated in terms of
a suitable Schwartz seminorm of~$f$:
\begin{equation} \label{eq:LREGR4}
 \| L_\sigma f \|~\leqslant~\| \check{f} \|_1^{}~\leqslant~
 (2\pi)^n \; \sum_{|\alpha|,|\beta| \leqslant 2n} \, \sup_{x \in V} \,
 |x^\alpha \, \partial_\beta^{} f (x)|~.
\end{equation}

To handle the case when $f$ belongs to~$\mathcal{B}(V)$, we need
a better estimate. Fortunately, we can resort to a famous theorem
from the theory of pseudo-differential operators, known as the
Calder\'on-Vaillancourt theorem~\cite{CV1}: in the version we
need here, which contemplates only a very special symbol class\,%
\footnote{The function space $\mathcal{B}$ coincides with
H\"ormander's symbol class space $S_{00}^0$.}
but on the other hand includes an improvement of the pertinent
estimate, taken from~\cite{Co1}, it states that, given any
totally bounded smooth function $a$ on $V \times V^*$, setting
\[
 (Au)(x)~=~\int_{V^*} d\xi~a(x,\xi) \, \check{u}(\xi) \;
           e^{i \langle \xi,x \rangle}
 \qquad \mbox{for $\, u \in \mathcal{S}(V)$}
\]
defines, by continuous linear extension, a bounded linear operator
on $L^2(V)$ with operator norm
\[
 \| A \|~\leqslant~
 C \; \sum_{|\alpha|,|\beta| \leqslant n} \; \sup_{x \in V ,\, \xi \in V^*} \,
 \bigl| \, \partial_{x,\alpha}^{} \partial_{\xi,\beta}^{} \>\!
           a \, (x,\xi) \, \bigr|~,
\]
where $C$ is a combinatorial constant depending only on the dimension~$n$
of~$V$.
Applying this result to the operator $L_\sigma f$ defined by equations~%
(\ref{eq:WMOY7}) and~(\ref{eq:LREGR0}), we see that for every $f$ in~%
$\mathcal{B}(V)$, $L_\sigma f$ is a bounded linear operator on $L^2(V)$
whose operator norm satisfies an estimate of the form
\begin{equation} \label{eq:LREGR5}
 \| L_\sigma f \|~\leqslant~
 |P(\sigma)| \; \sum_{|\alpha| \leqslant n} \; \sup_{x \in V} \;
 \bigl| \, \partial_\alpha^{} \>\! f (x) \, \bigr|~,
\end{equation}
where $P(\sigma)$ is a polynomial of degree $\leqslant n$ in $\sigma$
whose coefficients are combinatorial constants depending only on the
dimension~$n$ of~$V$.

From these results, it follows that we can define a $C^*$-norm on
$\mathcal{S}_\sigma$ as well as on $\mathcal{B}_\sigma$ by setting
\begin{equation}
 \|f\|~=~\|L_\sigma f\|~.
\end{equation}
That this is really a norm and not just a seminorm is due
to the fact that the left regular representation is faithful.
Namely, given $\, f \in \mathcal{B}(V)$, $f \neq 0$, and any
point~$x$ in~$V$ such that $\, f(x) \neq 0$, take $\, g \in
\mathcal{S}(V) \,$ such that $\, \check{g} \in \mathcal{S}%
(V^*) \,$ becomes
\[
 \check{g}(\xi)
 = \overline{f(x + {\textstyle \frac{1}{2}} \sigma^\sharp \xi)} \;
   e^{-i \langle \xi, x \rangle} \, e^{-q(\xi)}
\]
where $q$ is any positive definite quadratic form  on~$V^*$; then
by equation~(\ref{eq:WMOY7}), $(f \star_\sigma g)(x)$ is equal to
the $L^2$-norm of the function $\, \xi \mapsto f(x + \frac{1}{2}
\sigma^\sharp \xi) \,$ with respect to the Gaussian measure
$\, e^{-q(\xi)} \, d\xi \,$ on~$V^*$ and hence is $> 0$,
since this function is smooth and $\neq 0$ at $\xi=0$, so
$\, L_\sigma f \cdot g \neq 0 \,$ and hence $\, L_\sigma f \neq 0$.

The completions of~$\mathcal{S}_\sigma$ and of~$\mathcal{B}_\sigma$ with
respect to this $C^*$-norm will be denoted by $\mathcal{E}_\sigma$ and
by~$\mathcal{H}_\sigma$, respectively, and will be referred to as \emph%
{Heisenberg $C^*$-algebras}: more precisely, $\mathcal{E}_\sigma$ is the
\emph{nonunital Heisenberg $C^*$-algebra} while $\mathcal{H}_\sigma$ is
the \emph{unital Heisenberg $C^*$-algebra} (with respect to~$\sigma$).%
\footnote{The use of the symbol $\mathcal{E}$ at this point
is perhaps a bit unfortunate because $\mathcal{E}_\sigma$ has
nothing to do with the Schwartz space $\mathcal{E}(V)$ of
arbitrary smooth functions on~$V$: in fact, the difference
between the two becomes apparent by observing that when
$\sigma$ is nondegenerate, $\mathcal{E}_\sigma$ will be
isomorphic to the algebra of compact operators on the
Hilbert space $L^2(V_L)$ where $V_L$ is some lagrangian
subspace of~$V$.
Still, we have decided to adopt this notation for lack of
a better one and since it seems to have become traditional
in the area; moreover, the space $\mathcal{E}(V)$ will play
no role in this paper, except for an intermediate argument
in the appendix.}
Obviously, the estimate~(\ref{eq:LREGR4}) and the (much better)
estimate~(\ref{eq:LREGR5}) imply that the natural Fr\'echet
topologies on~$\mathcal{S}_\sigma$ and on~$\mathcal{B}_\sigma$
are finer than the $C^*$-topologies induced by their embeddings
into~$\mathcal{E}_\sigma$ and~$\mathcal{H}_\sigma$, respectively.
Moreover, by construction, the (faithful) $*$-representations~%
(\ref{eq:LREGR1}) and~(\ref{eq:LREGR2}) extend to (faithful)
$C^*$-representations of~$\mathcal{E}_\sigma$ and of~%
$\mathcal{H}_\sigma$, respectively, for which we maintain
the same notation, writing
\begin{equation} \label{eq:LREGR6}
 \begin{array}{cccc}
  L_\sigma: & \mathcal{E}_\sigma & \longrightarrow & B(L^2(V)) \\[1mm]
           &        f          &   \longmapsto   & L_\sigma f
 \end{array}~,
\end{equation}
and
\begin{equation} \label{eq:LREGR7}
 \begin{array}{cccc}
  L_\sigma: & \mathcal{H}_\sigma & \longrightarrow & B(L^2(V)) \\[1mm]
           &        f          &   \longmapsto   & L_\sigma f
 \end{array}~,
\end{equation}
respectively. It is also clear that the embedding of~$\mathcal{S}_\sigma$
into~$\mathcal{B}_\sigma$ (as a two-sided $*$-ideal) extends canonically to
an embedding of~$\mathcal{E}_\sigma$ into~$\mathcal{H}_\sigma$ (as an ideal)
and similarly that the embedding of equation~(\ref{eq:MULTA1}) extends
canonically to an embedding of~$\mathcal{H}_\sigma$ into the multiplier
algebra $M(\mathcal{E}_\sigma)$ of~$\mathcal{E}_\sigma$, which we shall
write in a form analogous to equation~(\ref{eq:MULTA1}):
\begin{equation} \label{eq:MULTA2}
 \begin{array}{ccc}
  \mathcal{H}_\sigma & \longrightarrow &  M(\mathcal{E}_\sigma) \\[1mm]
          f         &   \longmapsto   & (L_\sigma f,R_\sigma f) 
 \end{array}
\end{equation}
Moreover, the (faithful) $C^*$-representation of~$\mathcal{E}_\sigma$
in equation~(\ref{eq:LREGR6}) is also nondegenerate.%
\footnote{A $*$-representation of a $*$-algebra~$A$ by bounded operators
on a Hilbert space~$\mathfrak{H}$ is called nondegenerate if the subspace
generated by vectors of the form $\pi(a) \psi$, where $a \in A$ and $\psi
\in \mathfrak{H}$, is dense in~$\mathfrak{H}$, or equivalently, if there
is no nonzero vector in~$\mathfrak{H}$ that is annihilated by all
elements of~$A$. Obviously, if $A$ has a unit, every (unital)
$*$-representation of~$A$ is nondegenerate. Also, irreducible
$*$-representations and, more generally, cyclic $*$-representations
are always nondegenerate.}\linebreak
(This statement follows easily from the existence of approximate
identities in the Heisenberg-Schwartz algebra~$\mathcal{S}_\sigma$,
as formulated in Proposition~\ref{prp:APRIHS} of the appendix:
given any $L^2$-function $\, \psi \in L^2(V)$, it suffices to
approximate it in $L^2$-norm by some Schwartz function $\, f
\in \mathcal{S}(V) \,$ and then approximate that, in the
Schwartz space topology and hence also in $L^2$-norm, by
some Schwartz function of the form $\, \chi_k \star_\sigma f$,
where $\, \chi_k \in \mathcal{S}(V)$.)
Therefore, it extends uniquely to a (faithful) $C^*$-representation
\begin{equation} \label{eq:LREGR8}
 \begin{array}{cccc}
  L_\sigma: & M(\mathcal{E}_\sigma) & \longrightarrow & B(L^2(V)) \\[1mm]
           &           m          &   \longmapsto   & L_\sigma m
 \end{array}~,
\end{equation}
of its multiplier algebra $M(\mathcal{E}_\sigma)$. For later use,
we recall briefly how to define this extension: writing elements
of~$M(\mathcal{E}_\sigma)$ as pairs $\, m = (m_L,m_R) \,$ where
$\, m_L \in L(\mathcal{E}_\sigma) \,$ is a left multiplier
($m_L(f \star_\sigma g) = m_L(f) \star_\sigma g$) and
$\, m_R \in L(\mathcal{E}_\sigma) \,$ is a right multiplier
($m_R(f \star_\sigma g) = f \star_\sigma m_R(g)$), related by the
condition that $\, f \star_\sigma m_L(g) = m_R(f) \star_\sigma g$,
and using the fact that the representation $L_\sigma$ in equation~%
(\ref{eq:LREGR6}) is nondegenerate, which means that the subspace
of~$L^2(V)$ generated by vectors of the form $L_\sigma f \cdot \psi$
with $\, f \in \mathcal{E}_\sigma \,$ and $\, \psi \in L^2(V)$ (or
even $\, \psi \in \mathcal{S}(V)$) is dense in~$L^2(V)$, the
operator $\, L_\sigma m \in B(L^2(V))$ is defined by
\begin{equation} \label{eq:LREGR9}
 L_\sigma m \cdot \bigl( L_\sigma f \cdot \psi \bigr)~
 =~L_\sigma \bigl( m_L(f) \bigr) \cdot \psi~.
\end{equation}
That this is well-defined follows from the fact that $\mathcal{E}_\sigma$
is an essential ideal in~$M(\mathcal{E}_\sigma)$, i.e., an ideal that has
nontrivial intersection with any nontrivial ideal of~$M(\mathcal{E}_\sigma)$.
Moreover, it follows that, just like any $L_\sigma f$ (originally for
$\, f \in \mathcal{S}_\sigma \,$ but then, by continuity, also for
$\, f \in \mathcal{E}_\sigma$), any $L_\sigma m$ also commutes with
any $R_\sigma g$ (originally for $\, g \in \mathcal{S}_\sigma \,$
but then, by continuity, also for $\, g \in \mathcal{E}_\sigma$):
\begin{eqnarray*}
 (L_\sigma m \; R_\sigma g) \cdot \bigl( L_\sigma f \cdot \psi \bigr) \!\!
 &=&\!\! (L_\sigma m \; R_\sigma g) \cdot (f \star_\sigma \psi)~
  =~     L_\sigma m \cdot \bigl( (f \star_\sigma \psi) \star_\sigma g \bigr)
 \\[2mm]
 &=&\!\! L_\sigma m \cdot \bigl( f \star_\sigma (\psi \star_\sigma g) \bigr)~
  =~     L_\sigma m \cdot \bigl( L_\sigma f \cdot (\psi \star_\sigma g) \bigr)
 \\[2mm]
 &=&\!\! \bigl( L_\sigma (m_L(f)) \; R_\sigma g \bigr) \cdot \psi~
  =~     \bigl( R_\sigma g \; L_\sigma (m_L(f)) \bigr) \cdot \psi
 \\[2mm]
 &=&\!\! (R_\sigma g \; L_\sigma m) \cdot \bigl( L_\sigma f \cdot \psi \bigr)~.
\end{eqnarray*}
Finally, we see that with this construction, the representation
(\ref{eq:LREGR7}) becomes simply the composition of the representation
(\ref{eq:LREGR8}) with the embedding~(\ref{eq:MULTA2}).

\subsection{Uniqueness of the $C^*$-completion}

Having settled the question of existence of a $C^*$-norm on the
Fr\'echet $*$-algebras~$\mathcal{S}_\sigma$ and $\mathcal{B}_\sigma$,
we want to address the question of its uniqueness.
To this end, we follow the script laid out in the previous section,
which turns out to work perfectly for the Heisenberg-Schwartz and
Heisenberg-Rieffel algebras.

The first step will be to prove the following fact.
%
\begin{thm} \label{thm:SIHA} 
 The Heisenberg-Schwartz and Heisenberg-Rieffel algebras, $\mathcal{S}_\sigma$ 
 and $\mathcal{B}_\sigma$, are spectrally invariant in their respective $C^*$-%
 completions, $\mathcal{E}_\sigma$ and $\mathcal{H}_\sigma$, as defined above.
 Therefore, $\mathcal{E}_\sigma$ and $\mathcal{H}_\sigma$ are the universal
 enveloping $C^*$-algebras of the Heisenberg-Schwartz algebra~%
 $\mathcal{S}_\sigma$ and of the Heisenberg-Rieffel algebra~%
 $\mathcal{B}_\sigma$, respectively.
\end{thm}

\noindent
The assertion of Theorem~\ref{thm:SIHA} is known to hold in the commutative
case, i.e., when $\sigma=0$ \cite[Example~3.2, p.~135]{GBVF} and also when
$\sigma$ is nondegenerate \cite[Prop.~2.14 \& Prop.~2.23]{GGISV}, but for
the general deformed algebras it does not seem to have been stated explicitly
anywhere in the literature: in what follows, we shall give a different and
direct proof in which the rank of~$\sigma$ plays no role.
\begin{proof}
 The proof will be based on the main theorem of Ref.~\cite{MM},
 which can be formulated as follows. 
 To begin with, let $\Omega$ denote the standard symplectic form
 on the doubled space $\, V \oplus V^*$, defined by
 \begin{equation} \label{eq:HEISREP1}
  \Omega \bigl( (x,\xi) , (y,\eta) \bigr)~=~\xi(y) - \eta(x)~,
 \end{equation}
 let $H_\Omega$ denote the corresponding Heisenberg group,%
 \footnote{Note that the Heisenberg group $H_\Omega$ has nothing
 to do with the Heisenberg group $H_\sigma$ considered before.}
 and consider the corresponding strongly continuous unitary
 representation
 \begin{equation} \label{eq:HEISREP2}
  W_\Omega: H_\Omega~\longrightarrow~U(L^2(V))
 \end{equation}
 of $H_\Omega$ on~$L^2(V)$, explicitly given by
 \begin{equation} \label{eq:HEISREP3}
  \bigl( W_\Omega(x,\xi,\lambda) \, \psi \bigr) (z)~
  =~e^{-i \langle \xi ,\, z - \frac{1}{2} x \rangle + i \lambda} \, \psi(z-x)~.
 \end{equation}
 Next, consider the continuous isometric representation
 \begin{equation} \label{eq:HEISREP4}
  \mathrm{Ad}(W_\Omega): H_\Omega~\longrightarrow~\mathrm{Aut}(B(L^2(V)))
 \end{equation}
 of $H_\Omega$ on~$B(L^2(V))$ obtained from it by taking the adjoint action
 (i.e., for $\, T \in B(L^2(V))$, $\mathrm{Ad}(W_\Omega)(h) \, T = W_\Omega(h)
 \, T \, W_\Omega(h)^{-1}$).
 Then given an operator $\, T \in B(L^2(V))$, we say that it is
 \emph{Heisenberg-smooth} if it is a smooth vector with respect
 to this representation, i.e., if the function
 \[
  \begin{array}{ccc}
       H_\Omega    & \longrightarrow & B(L^2(V)) \\[1mm]
   (x,\xi,\lambda) &   \longmapsto   &
   W_\Omega(x,\xi,\lambda) \, T \, W_\Omega(x,\xi,\lambda)^{-1}
  \end{array}
 \]
 is smooth.
 Now the main theorem in~\cite{MM} states that an operator
 $\, T \in B(L^2(V)) \,$ is of the form $L_\sigma f$ (see
 equations (\ref{eq:WMOY7}),(\ref{eq:LREGR0}),(\ref{eq:LREGR2})),
 with $\, f \in \mathcal{B}_\sigma$, if and only if it is
 Heisenberg-smooth and commutes with all operators of
 the form $R_\sigma g$ (see equations (\ref{eq:WMOY6}),%
 (\ref{eq:RREGR0}),(\ref{eq:RREGR2})), where $\, g \in
 \mathcal{B}_\sigma$ (or equivalently, $g \in \mathcal{S}_\sigma$).
 This fact, applied in both directions, will enable us to complete
 the proof, as follows.

 Suppose first that $\, f \in \mathcal{B}_\sigma \,$ is invertible in~%
 $\mathcal{H}_\sigma$. Then the operator $\, L_\sigma f \in B(L^2(V)) \,$
 is Heisenberg-smooth and commutes with all operators of the form
 $R_\sigma g$, where $\, g \in \mathcal{S}_\sigma$. But this implies
 that the inverse operator $\, (L_\sigma f)^{-1} \in B(L^2(V)) \,$
 is also Heisenberg-smooth, since
 \[
  W_\Omega(x,\xi,\mu) \, (L_\sigma f)^{-1} \, W_\Omega(x,\xi,\mu)^{-1}~
  = \, \left( W_\Omega(x,\xi,\mu) \; L_\sigma f \;
              W_\Omega(x,\xi,\mu)^{-1} \right)^{-1}~,
 \]
 and since inversion of bounded linear operators is a smooth map, and
 that it also commutes with all operators of the form $R_\sigma g$, where
 $\, g \in \mathcal{S}_\sigma$. Thus it follows that $(L_\sigma f)^{-1}$
 is of the form $L_\sigma g$ for some $\, g \in \mathcal{B}_\sigma$,
 showing that $\mathcal{B}_\sigma$ is spectrally invariant in~%
 $\mathcal{H}_\sigma$.
 To prove that, similarly, $\mathcal{S}_\sigma$ is spectrally
 invariant in~$\mathcal{E}_\sigma$, consider the unitizations
 $\tilde{\mathcal{S}}_\sigma$ of~$\mathcal{S}_\sigma$ (still
 contained in~$\mathcal{B}_\sigma$) and $\tilde{\mathcal{E}}_\sigma$
 of~$\mathcal{E}_\sigma$ (still contained in~$\mathcal{H}_\sigma$),
 and suppose $\, f \in \mathcal{S}_\sigma \,$ to be such that
 $\, \lambda 1 + f \in \tilde{\mathcal{S}}_\sigma \,$ is invertible in~%
 $\tilde{\mathcal{E}}_\sigma$ (note that this implies $\, \lambda \neq 0$).
 Then, as we have already shown, $(\lambda 1 + L_\sigma f)^{-1}$ is of the
 form $L_\sigma h$ for some $\, h \in \mathcal{B}_\sigma$, which we can
 rewrite in the form $\, h = \lambda^{-1} 1 + g \,$ with $\, g \in
 \mathcal{B}_\sigma$, implying
 \[
  1~=~(\lambda 1+f) \star_\sigma (\lambda^{-1} 1 + g)~
    =~1 \, + \, \lambda^{-1} f \, + \, \lambda g \, + \, f \star_\sigma g
 \]
 and thus
 \[
  g~= \; - \, \lambda^{-2} f \, - \, \lambda^{-1} \, f \star_\sigma g~.
 \]
 But $\mathcal{S}_\sigma$ is an ideal in~$\mathcal{B}_\sigma$,
 so it follows that $\, g \in \mathcal{S}_\sigma \,$ and hence
 $\, \lambda^{-1} 1 + g \in \tilde{\mathcal{S}}_\sigma$.
\end{proof}

The same techniques can be used to prove the following interesting
and useful theorem about the relation between $\mathcal{E}_\sigma$
and $\mathcal{H}_\sigma$.
\begin{thm} \label{thm:HMULTA} 
 The $C^*$-algebra~$\mathcal{H}_\sigma$ is the multiplier algebra of
 the $C^*$-algebra~$\mathcal{E}_\sigma$,
 \begin{equation} \label{eq:HMULTA}
  \mathcal{H}_\sigma~=~M(\mathcal{E}_\sigma)~,
 \end{equation}
 and in fact it is a von Neumann algebra.
\end{thm}

\begin{proof}
 What needs to be shown is that the embedding~(\ref{eq:MULTA2})
 is in fact an isomorphism. \linebreak To this end, let $R$ be
 the subspace of~$B(L^2(V))$ consisting of right translations by
 elements of $\mathcal{S}_\sigma$:
 \[
   R = \{ R_\sigma(g) \; | \; g \in \mathcal{S}_\sigma \}~.
 \] 
 What will be of interest here is its commutant $R'$, which is a closed
 subspace (and in fact even a von Neumann subalgebra) of~$B(L^2(V))$.
 As has been shown at the end of the previous subsection, the
 representation~(\ref{eq:LREGR8}) maps $M(\mathcal{E}_\sigma)$
 into~$R'$.
 On the other hand, the relation
 \begin{equation} \label{eq:RREGRx}
  W_\Omega(x,\xi,\lambda) \, R_\sigma(g) \, W_\Omega(x,\xi,\lambda)^{-1}~
  =~R_\sigma
    \bigl( W_\Omega(x + {\textstyle \frac{1}{2}} \sigma^\sharp \xi,0,0) \, g
           \bigr)~,
 \end{equation}
 shows that $R$ is an invariant subspace for the representation
 $\mathrm{Ad}(W_\Omega)$ of~$H_\Omega$ on $B(L^2(V))$ (see equation~%
 (\ref{eq:HEISREP4})); hence so is $R'$. Therefore, the main theorem
 of Ref.~\cite{MM} can be reformulated as the statement that the image
 of~$\mathcal{B}_\sigma$ under the representation~(\ref{eq:LREGR2}) is
 precisely the subspace of smooth vectors for the representation
 $\mathrm{Ad}(W_\Omega)$ of~$H_\Omega$ on~$R'$ obtained by restriction,
 and hence it is dense in~$R'$. It follows that the image of~%
 $\mathcal{H}_\sigma$ under the representation~(\ref{eq:LREGR7})
 is precisely $R'$, a von Neumann algebra.
\end{proof}

For the second step, we use a result that is of independent interest,
namely the fact that, as shown in~\cite[Proposition~5.2]{RDQ},
the Heisenberg $C^*$-algebra $\mathcal{E}_\sigma$ is isomorphic to
the algebra of continuous functions, vanishing at infinity, on a
certain subspace $V_0^{}$ of~$V$ (dual to $\ker \sigma$) and taking
values in the algebra $\mathcal{K}$ of compact linear operators in
a separable Hilbert space, which can also be written as a $C^*$-%
algebra tensor product:
\begin{equation} \label{eq:TP1} 
 \mathcal{E}_\sigma^{}~\cong~C_0^{}(V_0^{},\mathcal{K})~
                      \cong~C_0^{}(V_0^{}) \otimes \mathcal{K}~.
\end{equation}
To see this explicitly, we first recall the ``musical homomorphism''
$\, \sigma^\sharp: V^* \longrightarrow V \,$ induced by $\sigma$ (i.e.,
$\langle \xi , \sigma^\sharp \eta \rangle = \sigma(\eta,\xi)$) whose
image is a subspace of~$V$ that we shall denote by~$W$: it is
precisely the annihilator of the kernel of~$\sigma$ in~$V^*$,
\begin{equation} \label{eq:SYMPS1} 
 W~=~\mbox{im} \, \sigma^\sharp~=~(\ker \sigma)^\perp~,
\end{equation}
and it carries a symplectic form, denoted by~$\omega$ and defined by
$\, \omega (\sigma^\sharp \xi,\sigma^\sharp \eta) = \sigma(\xi,\eta)$.
Now choosing a subspace~$V_0^{}$ of~$V$ complementary to~$W$, we get
a direct decomposition
\begin{equation} \label{eq:DECOM1}
 V~=~V_0^{} \oplus W~.
\end{equation}
Taking the corresponding annihilators, we also get a direct
decomposition for the dual,
\begin{equation} \label{eq:DECOM2}
 V^*~=~V_0^* \oplus W^*
 \qquad \mbox{where} \qquad
 V_0^* = W^\perp = \ker \sigma
 \quad \mbox{and} \quad
 W^* = V_0^\perp~.
\end{equation}
Of course, $W^*$ also carries a symplectic form, again denoted by~%
$\omega$, which is simply the restriction of~$\sigma$ to this subspace,
on which it is nondegenerate.
Now according to the Schwartz nuclear theorem, we have
\begin{equation} \label{eq:TP2} 
 \mathcal{S}(V)~\cong~\mathcal{S}(V_0^{}) \otimes \mathcal{S}(W)~,
\end{equation}
and similarly
\begin{equation} \label{eq:TP3} 
 \mathcal{S}(V^*)~\cong~\mathcal{S}(V_0^*) \otimes \mathcal{S}(W^*)~,
\end{equation}
and it is clear that the Fourier transform $\, \mathcal{F}:
\mathcal{S}(V) \longrightarrow \mathcal{S}(V^*) \,$ is the tensor
product of the Fourier transforms $\, \mathcal{F}: \mathcal{S}(V_0^{})
\longrightarrow \mathcal{S}(V_0^*) \,$ and $\, \mathcal{F}:
\mathcal{S}(W) \longrightarrow \mathcal{S}(W^*)$.
Hence looking at the definition of the Weyl-Moyal star product,
we see that the tensor products in equations~(\ref{eq:TP2}) and~%
(\ref{eq:TP3}) are in fact tensor products of algebras, i.e.,
\begin{equation} \label{eq:TP4} 
 \mathcal{S}_\sigma^{}~\cong~\mathcal{S}_0^{} \otimes \mathcal{S}_\omega^{}~,
\end{equation}
where $\mathcal{S}_0^{}$ is the commutative algebra of Schwartz test
functions ($\mathcal{S}(V_0^{})$ with the ordinary pointwise product
or $\mathcal{S}(V_0^*)$ with the ordinary convolution product) while
$\mathcal{S}_\omega^{}$ is the Heisenberg-Schwartz algebra associated
with the nondegenerate $2$-form $\omega$.
Taking the universal $C^*$-completions, we get 
\begin{equation} \label{eq:TP5} 
 \mathcal{E}_\sigma^{}~\cong~\mathcal{E}_0^{} \otimes \mathcal{E}_\omega^{}~.
\end{equation}
But obviously, $\mathcal{E}_0^{} \cong C_0^{}(V_0^{}) \cong C_0^{}(V_0^*)$,
and it is well known that $\, \mathcal{E}_\omega^{} \cong \mathcal{K}$.

In passing, we note that the tensor product in equations~(\ref{eq:TP1})
and~(\ref{eq:TP5}) is the tensor product of $C^*$-algebras and as such
is unique (there is only one $C^*$-norm on the algebraic tensor product)
since one of the factors is nuclear (in fact, both of them are; see
\linebreak \cite[Example~6.3.2 \& Theorem~6.4.15]{MU}).

To complete the argument, we make use of the fact that any ideal in
$\mathcal{E}_\sigma^{}$ is of the form
\[
 \bigl\{ \phi \in C_0^{}(V_0^{},\mathcal{K})~|~\phi|_F^{} = 0 \bigr\}~,
\]
or equivalently,
\[
 \bigl\{ f \in C_0^{}(V_0^{})~|~f|_F^{} = 0 \bigr\}
 \otimes \mathcal{K}~,
\]
where $F$ is a closed subset of the space $V_0^{}$. (That these are
in fact all ideals in $\mathcal{E}_\sigma^{}$ is a special case of a
much more general statement, whose formulation and proof can be
found in~\cite[Lemma~VIII.8.7]{FD}, together with the fact that
$\mathcal{K}$ is simple.)
But obviously, each of these ideals has nontrivial intersection
with the Heisenberg-Schwartz algebra.

Finally, we can extend the conclusion from $\mathcal{E}_\sigma$
to~$\mathcal{H}_\sigma$: since the latter is the multiplier algebra
of the former, any nontrivial ideal of~$\mathcal{H}_\sigma$ intersects
$\mathcal{E}_\sigma$ in a nontrivial ideal of~$\mathcal{E}_\sigma$, which
in turn has nontrivial intersection with~$\mathcal{S}_\sigma$ and hence
with~$\mathcal{B}_\sigma$.

Summarizing, we have proved
\begin{thm} \label{thm:UNCN} 
 The Heisenberg-Schwartz algebra $\mathcal{S}_\sigma$ and the
 Heisenberg-Rieffel algebra $\mathcal{B}_\sigma$ each admit one
 and only one $C^*$-norm, and hence the Heisenberg $C^*$-algebras
 $\mathcal{E}_\sigma$ and~$\mathcal{H}_\sigma$ are their unique
 $C^*$-completions.
\end{thm}

\subsection{Representation Theory}

Returning to the situation discussed at the beginning of this section,
assume we are given any strongly continuous unitary representation $\pi$
of the Heisenberg group~$H_\sigma$.
Then Weyl quantization produces a $*$-representation $W_\pi$ of the
Heisenberg-Schwartz algebra $\mathcal{S}_\sigma$, defined according
to equations~(\ref{eq:WQUANT1})-(\ref{eq:WQUANT2}), which according
to equation~(\ref{eq:SCHWEST1}) is continuous with respect to the
Schwartz topology. But in fact it is also continuous with respect
to the $C^*$-topology since that is defined by the maximal $C^*$-%
norm on~$\mathcal{S}_\sigma$ which is an upper bound for all $C^*$-%
seminorms on~$\mathcal{S}_\sigma$, including the operator seminorm
for $W_\pi$, and therefore $W_\pi$ extends uniquely to a $C^*$-%
representation of the nonunital Heisenberg $C^*$-algebra
$\mathcal{E}_\sigma$ which will again be denoted by~$W_\pi$.
Moreover, we have
\begin{lem} \label{lem:NONDEG}
 Given any strongly continuous unitary representation $\pi$ of the
 Heisenberg group~$H_\sigma$, the resulting $*$-representation $W_\pi$
 of the Heisenberg-Schwartz algebra $\mathcal{S}_\sigma$, and hence also
 of the Heisenberg $C^*$-algebra $\mathcal{E}_\sigma$, is nondegenerate.%
 \footnotemark[9]
\end{lem}
\begin{proof}
 Given any vector $\psi$ in the Hilbert space $\mathfrak{H}$ of
 the representations~$\pi$ and $W_\pi$ and any $\, \epsilon > 0 \,$,
 strong continuity of~$\pi$ implies the existence of an open
 neighborhood $U^*$ of~$0$ in~$V^*$ such that
 \[
  \| \pi(\xi) \psi - \psi \|~<~\epsilon
  \qquad \mbox{for $\, \xi \in U^*$}~,
 \]
 since $\, \pi(0) = 1$. Now choose $\, f \in \mathcal{S}(V) \,$
 such that $\, \check{f} \in \mathcal{S}(V^*) \,$ is nonnegative,
 with integral normalized to~$1$, and has compact support
 contained in~$U^*$. Then
 \[
  \| (W_\pi f) \psi - \psi \|~
      =    ~\biggl\| \, \int_{V^*} d\xi~\check{f}(\xi) \; \pi(\xi) \psi \,
                     - \, \psi \, \biggr\|~
  \leqslant~\int_{V^*} d\xi~\check{f}(\xi)~\| \pi(\xi) \psi - \psi \|~
      <    ~\epsilon~.
 \]
\end{proof}

\noindent
As a result, these $*$-representations extend to (unital) $*$-%
representations of the Heisenberg-Rieffel algebra $\mathcal{B}_\sigma$
and of the unital Heisenberg $C^*$-algebra $\mathcal{H}_\sigma$,
respectively, which will again be denoted by~$W_\pi$.

Conversely, given any nondegenerate $C^*$-representation $W$ of~%
$\mathcal{E}_\sigma$, we can extend it uniquely to a (unital) $C^*$-%
representation of~$\mathcal{H}_\sigma$, again denoted by~$W$, which
restricts to a unitary representation $\pi_W$ of~$H_\sigma$ defined
according to
\begin{equation}
 \pi_W(\xi)~=~W \bigl(e_\xi)~,
\end{equation}
where $\, e_\xi \in \mathcal{B}_\sigma \,$ denotes the phase function
given by $\, e_\xi(v) = e^{i \langle \xi,v \rangle}$.
To show that $\pi_W$ is automatically strongly continuous, we note that,
according to equations~(\ref{eq:FOURT1}) and~(\ref{eq:WMOY7}), we have,
for any $\, f \in \mathcal{S}_\sigma$,
\[
 (e_\xi \star_\sigma f)(x)~
 =~e^{i \langle \xi,x \rangle} \, f(x - {\textstyle \frac{1}{2}} \sigma^\sharp \xi)~,
\]
so $e_\xi \star_\sigma f$ converges to $f$ as $\xi$ tends to zero,
in the Schwartz topology and hence also in the $C^*$-topology.
Now since $W$ is supposed to be nondegenerate and $\mathcal{S}_\sigma$
is dense in~$\mathcal{E}_\sigma$, every vector in $\mathfrak{H}_W$ can
be approximated by vectors of the form $W(f)\psi$ where $\, f \in
\mathcal{S}_\sigma$ \linebreak and $\, \psi \in \mathfrak{H}_W$.
But on such vectors, we have strong continuity, since for any
$\, f \in \mathcal{S}_\sigma \,$ and any $\, \psi \in \mathfrak{H}_W$,
$\, \pi_W(\xi) W(f)\psi = W(e_\xi \star_\sigma f) \psi \,$ tends to
$W(f)\psi$ as $\xi$ tends to zero.

Finally, it is easy to see that composing the two operations of passing
(a) from a strongly continuous unitary representation $\pi$ of~$H_\sigma$
to a nondegenerate $C^*$-representation $W_\pi$ of~$\mathcal{E}_\sigma$ and
(b) from a nondegenerate $C^*$-representation $W$ of~$\mathcal{E}_\sigma$
to a strongly continuous unitary representation $\pi_W$ of~$H_\sigma$, in
any order, reproduces the original representation, so we have proved
\begin{thm}[Correspondence theorem] \label{thm:REP1}
 There is a bijective correspondence between the strongly continuous
 unitary representations of the Heisenberg group~$H_\sigma$ and the
 non\-degenerate $C^*$-representations of the nonunital Heisenberg
 $C^*$-algebra~$\mathcal{E}_\sigma$. Moreover, this correspondence
 takes irreducible representations to irreducible representations.
\end{thm}

As a corollary, we can state a classification theorem for irreducible
representations which is based on one of von Neumann's famous theorems,
according to which there is a \emph{unique} such representation,
generally known as the \emph{Schr\"odinger representation of the
canonical commutation relations}, provided that $\sigma$ is
nondegenerate.
To handle the degenerate case, i.e., when $\sigma$ has a nontrivial
null space, denoted by $\, \ker \sigma$, we use the same trick as above:
choose a subspace $W^*$ of~$V^*$ complementary to $\, \ker \sigma$ (see
equation~(\ref{eq:DECOM2})), so that the restriction $\omega$ of~$\sigma$
to~$\, W^* \times W^* \,$ is nondegenerate, and introduce the corresponding
Heisenberg algebra $\, \mathfrak{h}_\omega = W^* \oplus \mathbb{R} \,$ and
Heisenberg group $\, H_\omega = W^* \times \mathbb{R} \,$ to decompose the
original ones into the direct sum $\, \mathfrak{h}_\sigma = \ker \sigma
\oplus \mathfrak{h}_\omega \,$ of two commuting \mbox{ideals} and $\, H_\sigma
= \ker \sigma \times H_\omega \,$ of two commuting normal subgroups.%
\footnote{As is common practice in the abelian case, we consider
the same vector space $\, \ker \sigma$ as an abelian Lie algebra
in the first case and as an (additively written) abelian Lie group
in the second case, so that the exponential map becomes the identity.}
It follows that every (strongly continuous unitary) representation
of~$H_\sigma$ is the tensor product of a (strongly continuous unitary)
representation of~$\, \ker \sigma$ and a (strongly continuous unitary)
representation of~$H_\omega$, where the first is irreducible if and
only if each of the last two is irreducible.
Now since $\, \ker \sigma$ is abelian, its irreducible representations
are one-dimensional and given by their character, which proves the
following
\begin{thm}[Classification of irreducible representations] \label{thm:IRREP1}
 With the notation above, the strongly continuous, unitary, irreducible
 representations of the Heisenberg group~$H_\sigma$, or equivalently, the
 irreducible representations of the nonunital Heisenberg $C^*$-algebra
 $\mathcal{E}_\sigma$, are classified by their \textbf{highest weight},
 which is a vector $v$ in~$V$, or more precisely, its class~$[v]$ in
 the quotient space $V/(\ker \sigma)^\perp$, such that
 \[
  \pi_{[v]}(\xi,\eta)~=~e^{i\langle\xi,v\rangle} \, \pi_\omega(\eta)
  \qquad \mathrm{for}~~\xi \in \ker \sigma \,,\,
         \eta \in H_\omega~,
 \]
 where $\pi_\omega$ is of course the Schr\"odinger representation of~%
 $H_\omega$.
\end{thm}

It may be worthwhile to point out that the correspondence of Theorem~%
\ref{thm:REP1} does \emph{not} hold when we replace $\mathcal{E}_\sigma$
by~$\mathcal{H}_\sigma$, simply because $\mathcal{H}_\sigma$ admits $C^*$-%
representations whose restriction to~$\mathcal{E}_\sigma$ is trivial:
just consider any representation of the corona algebra $\mathcal{H}_\sigma/%
\mathcal{E}_\sigma$. That is why it is important to consider not only
$\mathcal{H}_\sigma$ but also~$\mathcal{E}_\sigma$.

To conclude this section, we would like to comment on the difference between
our definition of the Heisenberg $C^*$-algebra and others that can be found
in the literature~-- more specifically, the Weyl algebra $\, \overline%
{\Delta(V^*,\sigma)} \,$ of Refs~\cite{MA,MSTV} and the resolvent algebra
$\, \mathcal{R}(V^*,\sigma) \,$ of Ref.~\cite{BG1}: these are defined as
the universal enveloping $C^*$-algebras of the $*$-algebra $\, \Delta%
(V^*,\sigma) \,$ generated by the phase functions $e_\xi$ and of the
$*$-algebra $\, \mathcal{R}_0(V^*,\sigma) \,$ generated by the resolvent
functions~$R_\xi$, respectively, where $\, e_\xi(v) = e^{i \langle \xi,v \rangle}$,
as before, and similarly, $R_\xi(v) = (i - \langle \xi,v \rangle)^{-1}$.

The main problem with these constructions is that the resulting $C^*$-%
algebras are, in a certain sense, ``too small'', as indicated by the fact
that they accomodate lots of ``purely algebraic'' representations
and one has to restrict to a suitable class of ``regular'' representations in
order to establish a bijective correspondence with the usual representations
of the CCRs: nonregular representations do not even allow to define the
``infinitesimal'' operators that would be candidates for satisfying the CCRs.
Moreover, the choice of the respective generating $*$-subalgebras
$\, \Delta(V^*,\sigma) \,$ and $\, \mathcal{R}_0(V^*,\sigma) \,$ is
to a certain extent arbitrary, and even though they admit maximal
$C^*$-norms, they do \emph{not} in general admit a \emph{unique}
$C^*$-norm.
What is remarkable about the extensions proposed here, using the larger $C^*$-%
algebras $\mathcal{E}_\sigma$ or $\mathcal{H}_\sigma$, together with the larger
generating $*$-subalgebras $\mathcal{S}_\sigma$ or $\mathcal{B}_\sigma$, is that
this procedure eliminates the unwanted representations (whose inclusion would
invalidate the analogue of Theorem~\ref{thm:IRREP1} classifying the irreducible
representations) as well as the ambiguity in the choice of $C^*$-norm.

On the other hand, it must be emphasized that our approach is restricted
to the case of finite-dimensional Poisson vector spaces (quantum mechanics):
the question of whether, and how, it is possible to extend it to infinite-%
dimensional situations (quantum field theory) is presently completely open.

\section{Bundles of $*$-Algebras and \emph{C}$^*$-Completions}
\label{sec:*bun}

In the present section we want to introduce concepts that will allow
us to extend the process of $C^*$-completion of $*$-algebras discussed
in Section~\ref{sec:*alg} to bundles of $*$-algebras.

To begin with, we want to digress for a moment to briefly discuss an
important question of terminology, generated by the indiscriminate and
confusing use of the term ``bundle'' in the literature on the subject.

Assume that $(V_x)_{x \in X}$ is a family of sets indexed by the points
$x$ of some other set~$X$.
Then we may introduce the set $V$ defined as their disjoint union,
\begin{equation} \label{eq:FIBR1}
 V~=~\bigcup_{x \in X}^{\scriptscriptstyle{\bullet}} V_x~,
\end{equation}
together with the surjective map $\, \rho: V \longrightarrow X \,$
that takes $V_x$ to $x$: this defines a \emph{``bundle''}\/ with
\emph{total space}\/~$V$, \emph{base space}\/~$X$ and \emph%
{projection}\/~$\rho$, with $\, V_x = \rho^{-1}(x) \,$ as the
\emph{fiber over}\/ the point~$x$.
The question is what additional conditions should be imposed on this
kind of structure in order to allow us to remove the quotation marks
on the expression ``bundle''.
For example, in the context of topology, it is usually required that
both $V$ and~$X$ should be topological spaces and that $\rho$ should
be continuous and open.
Similarly, in the context of differential geometry, one requires
that, in addition, both $V$ and~$X$ should be manifolds and that
$\rho$ should be a submersion.
Of course, special care must be taken when these manifolds are
infinite-dimensional, since dealing with these is a rather touchy
business; in particular, the standard theory that works in the
context of Banach spaces and manifolds, for which we may refer
to~\cite{LA}, does not apply to more generally locally convex
spaces and manifolds, for which one must resort to more
sophisticated techniques such as the ``convenient calculus''
developed in~\cite{KM}.

Within this context, a central role is played by the condition of
\emph{local triviality}\/, which requires the existence of a fixed
topological space or of a fixed manifold $V_0$, called the \emph%
{typical fiber}\/, and of some covering of the base space by open
subsets such that for each one of them, say~$U$, the subset
$\rho^{-1}(U)$ of the total space is homeomorphic (in the case
of topological spaces) or diffeomorphic (in the case of manifolds)
to the cartesian product \mbox{$\, U \times V_0 \,$:} in this case,
one says that $V$ is a \emph{fiber bundle} over~$X$ and refers to
the afore\-mentioned homeomorphisms or diffeomorphisms as \emph%
{local trivializations}\/. When $V_0$ and each of the fibers $V_x$
($x \in X$) come with a certain (fixed) type of additional
structure and local trivializations can be found which preserve
that structure, an appropriate reference is incorporated into the
terminology: for example, one says that $V$ is a \emph{vector bundle}
over~$X$ when $V_0$ and each of the fibers $V_x$ ($x \in X$) are vector
spaces and local trivializations can be chosen to be fiberwise linear.
Thus the standard terminology used in topology and differential geometry
suggests that fiber bundles, vector bundles etc.~-- and in particular,
$C^*$-algebra bundles~-- should be locally trivial.

Unfortunately, this convention is not followed universally. In particular,
in the theory of operator algebras, it is necessary to allow for a greater
degree of generality, since there appear important examples where local
triviality fails and where even some of the structure maps may fail to
be continuous. Therefore, let us state explicitly what is required:
\begin{dfn} \label{def:BUNALG}
 A \textbf{bundle of locally convex $*$-algebras} over a locally compact
 topo\-logical space~$X$ is a topological space $\mathcal{A}$ together with a
 surjective continuous and open map $\, \rho: \mathcal{A} \longrightarrow X$,
 equipped with the following structures: (a) operations of fiberwise addition,
 scalar multiplication, multiplication and involution that turn each fiber
 $\, \mathcal{A}_x = \rho^{-1}(x) \,$ into a $*$-algebra and are such that
 the corresponding maps
 \[
  \begin{array}{ccc}
   \mathcal{A} \times_X \mathcal{A} & \longrightarrow & \mathcal{A} \\[1mm]
               (a_1,a_2)             &   \longmapsto   &  a_1 + a_2
  \end{array} \quad , \quad
  \begin{array}{ccc}
   \mathbb{C} \times \mathcal{A} & \longrightarrow & \mathcal{A} \\[1mm]
             (\lambda,a)         &   \longmapsto   &  \lambda a
  \end{array}
 \]
 and
 \[
  \begin{array}{ccc}
   \mathcal{A} \times_X \mathcal{A} & \longrightarrow & \mathcal{A} \\[1mm]
               (a_1,a_2)             &   \longmapsto   &   a_1 a_2
  \end{array} \quad , \quad
  \begin{array}{ccc}
   \mathcal{A} & \longrightarrow & \mathcal{A} \\[1mm]
        a      &   \longmapsto   &     a^*
  \end{array}
 \]
 where $\, \mathcal{A} \times_X \mathcal{A} = \{ (a_1,a_2) \in
 \mathcal{A} \times \mathcal{A}~|~\rho(a_1) = \rho(a_2) \} \,$ is the
 fiber product of~$\mathcal{A}$ with itself over~$X$, are continuous,%
 \footnote{Actually, a simple generalization of an argument that
 can be found in~\cite[Proposition~C.17, p.~361]{Wil} shows that
 it is sufficient to require that scalar multiplication should be
 continuous in the second variable, i.e., for each $\, \lambda \in
 \mathbb{C}$, the map $\, \mathcal{A} \longrightarrow \mathcal{A}$,
 $a \longrightarrow \lambda a \,$ is continuous: this condition is
 often easier to check in practice, but it already implies joint
 continuity.}
 and (b) a directed set $\Sigma$ of nonnegative functions\,%
 \footnote{The condition of $\Sigma$ being directed refers to the natural
 order on the set of all nonnegative functions on~$\mathcal{A}$, defined
 pointwise.}
 $\, s: \mathcal{A} \longrightarrow \mathbb{R}$ \linebreak which,
 at every point $x$ in $X$, provides a directed set $\, \Sigma_x
 = \{ s|_{\mathcal{A}_x} \,|\, s \in \Sigma \} \,$ of seminorms
 on the fiber $\, \mathcal{A}_x = \rho^{-1}(x) \,$ turning it
 into a locally convex $*$-algebra; we shall refer to the
 functions $s$ in~$\Sigma$ as \textbf{fiber seminorms} on~%
 $\mathcal{A}$.
 Moreover, when each of these fiber seminorms is either continuous
 or else just upper semicontinuous, and when taken together they
 \mbox{satisfy} the additional continuity condition that any net
 $(a_i)_{i \in I}$ in~$\mathcal{A}$ such that $\, s(a_i) \to 0$
 \linebreak
 for every $\, s \in \Sigma \,$ and $\, \rho(a_i) \to x \,$ for
 some $\, x \in X \,$ actually converges to $\, 0_x \in \mathcal{A}_x$,
 then we say that $\mathcal{A}$ is either a \textbf{continuous} or else
 an \textbf{upper semicontinuous bundle} of locally convex $*$-algebras,
 respectively.
 Finally, we shall say that such a bundle $\mathcal{A}$ is
 \textbf{unital} if all of its fibers $\mathcal{A}_x$ are
 $*$-algebras with unit and, in addition, the unit section
 \[
  \begin{array}{ccc}
   X & \longrightarrow & \mathcal{A} \\[1mm]
   x &   \longmapsto   &     1_x
  \end{array}
 \]
 is continuous. Special cases are:
 \begin{itemize}
  \item $\mathcal{A}$ is a \textbf{bundle of Fr\'echet $*$-algebras}
        if $\Sigma$ is countable and each fiber is complete in the
        induced topology: in this case, $\Sigma$ can (and will) be
        arranged in the form of an increasing sequence.
  \item $\mathcal{A}$ is a \textbf{bundle of Banach $*$-algebras}
        if $\Sigma$ is finite and each fiber is complete in the
        induced topology: in this case, $\Sigma$ can (and will)
        be replaced by a single function $\, \|.\| : \mathcal{A}
        \longrightarrow \mathbb{R}$, called the fiber norm, which
        induces a Banach $*$-algebra norm on each fiber.
  \item $\mathcal{A}$ is a \textbf{bundle of C\,$^*$-algebras},
        or simply \textbf{C\,$^*$-bundle}, if it is a bundle of
        Banach $*$-algebras whose fiber norm induces a $C^*$-%
        norm on each fiber.
 \end{itemize}
\end{dfn}

\noindent
In particular, according to the convention adopted in this paper, bundles
of $*$-algebras over~$X$ need not be locally trivial and hence the property
of local triviality~-- either in the sense of topology when $X$ is a
topological space (continuous transition functions) or in the sense
of differential geometry when $X$ is a manifold (smooth transition
functions)~-- will be stated explicitly when it is satisfied and relevant.

On a historical side, we note that a first version of this definition
was formulated by \linebreak Dixmier~\cite{Dix}, through his notion of
a ``continuous field of $C^*$-algebras''. Somewhat later, Fell introduced
the concept of a continuous $C^*$-bundle (see, e.g., \cite[Definition~8.2,
p.~580]{FD}), providing an (ultimately) equivalent but intuitively more
appealing approach. Finally, it was observed that most of the important
results continue to hold with almost no changes for upper semicontinuous
$C^*$-bundles, the main difference being that in this case, the total
space $\mathcal{A}$ may fail to be Hausdorff. The extension proposed
here, to bundles whose fibers are more general locally convex $*$-%
algebras (of various types), seems natural and will be useful for
what follows.

The additional continuity condition formulated in the above definition
guarantees that the topology on the total space~$\mathcal{A}$ is uniquely
determined by the set of fiber seminorms~$\Sigma$; this follows directly
from the following generalization of a theorem of Fell:
\begin{thm}[Topology of $*$-bundles] \label{thm:FELL1}
 Assume that $(\mathcal{A}_x)_{x \in X}$ is a family of $*$-algebras
 indexed by the points $x$ of a locally compact topological space~$X$,
 and consider the disjoint union
 \begin{equation} \label{eq:STAFIB1}
  \mathcal{A}~= \bigcup_{x \in X}^{\scriptscriptstyle{\bullet}} \mathcal{A}_x
 \end{equation}
 as a ``bundle'' over~$X$ (in the purely set-theoretical sense).
 Assume further that $\Sigma$ is a directed set of fiber seminorms
 on~$\mathcal{A}$ turning each fiber $\mathcal{A}_x$ of~$\mathcal{A}$
 into a locally convex $*$-algebra (Fr\'echet $*$-algebra\,/\,%
 Banach $*$-algebra\,/\,$C^*$-algebra) and that $\Gamma$ is a
 $*$-algebra of sections of this ``bundle'', satisfying the
 following properties:
 \begin{enumerate}[(a)]
  \item For each section $\, \varphi \in \Gamma \,$ and each fiber seminorm
        $\, s \in \Sigma$, the function $\, X \longrightarrow \mathbb{R}$,
        $x \longmapsto s(\varphi(x)) \,$ is upper semicontinuous (or
        continuous).
  \item For each point $x$ in~$X$, the $*$-subalgebra $\, \Gamma_x
        = \{ \varphi(x) \, | \, \varphi \in \Gamma \} \,$ of~%
        $\mathcal{A}_x$ is dense in~$\mathcal{A}_x$.
 \end{enumerate}
 Then there is a unique topology on~$\mathcal{A}$ turning it into an upper
 semicontinuous (or continuous) bundle of locally convex $*$-algebras
 (Fr\'echet $*$-algebras\,/\,Banach $*$-algebras\,/\,$C^*$-algebras)
 over~$X$, repectively, such that $\Gamma$ becomes a $*$-subalgebra
 of the $*$-algebra $\Gamma(X,\mathcal{A})$ of all continuous sections
 of~$\mathcal{A}$.
\end{thm}

\noindent
Similar statements can be found, e.g., in~\cite[Theorem II.13.18]{FD}
(for continuous bundles of Banach spaces) and in \cite[Theorem~C.25,
p.~364]{Wil} (for upper semicontinuous bundles of $C^*$-algebras),
but the proof is easily adapted to the more general situation
considered here; in particular, a basis of the desired topology
on~$\mathcal{A}$ is given by the subsets
\[
 W(\varphi,U,s,\epsilon)~
 =~\big\{ a \in \mathcal{A} \; | \; \rho(a) \in U \, , \,
          s(a - \varphi(\rho(a))) < \epsilon \big\}~,
\]
where $\, \rho: \mathcal{A} \longrightarrow X \,$ is the bundle projection,
$\varphi \in \Gamma$, $U$ is an open subset of~$X$, $s \in \Sigma$ \linebreak
and~$\, \epsilon > 0$.

Whatever may be the specific class of bundles considered, the notion of
morphism between them is the natural one:
\begin{dfn} \label{def:BUNMOR}
 Given two bundles of locally convex $*$-algebras\/ $\mathcal{A}$ and~%
 $\mathcal{B}$ over locally compact topological spaces\/ $X$ and\/~$Y$,
 respectively, a \textbf{bundle morphism} from~$\mathcal{A}$ to~%
 $\mathcal{B}$ is a continuous map $\, \phi: \mathcal{A} \longrightarrow
 \mathcal{B} \,$ which is fiber preserving in the sense that there exists a
 (necessarily unique) continuous map $\, \check{\phi}: X \longrightarrow Y \,$
 such that the following diagram commutes,
 \[
  \begin{minipage}{5cm}
   \xymatrix{
    \mathcal{A} \ar[r]^{\phi} \ar[d]_{\rho_{\mathcal{A}}} &
    \mathcal{B} \ar[d]^{\rho_{\mathcal{B}}} \\
    X \ar[r]^{\check{\phi}} & Y
   }
  \end{minipage}
 \]
 and such that for every point\/ $x$ in\/~$X$,
 the restriction $\, \phi_x^{}: \mathcal{A}_x^{}
 \longrightarrow \mathcal{B}_{\check{\phi}(x)} \,$ of~$\phi$
 to the fiber over~$x$ is a homomorphism of locally convex
 $*$-algebras.
 When $\, Y=X \,$ and $\check{\phi}$ is the identity, we say
 that $\phi$ is \textbf{strict} (over~$X$).
\end{dfn}

Theorem~\ref{thm:FELL1} above already makes it clear that an important
object associated with any upper semicontinuous bundle $\mathcal{A}$ of
locally convex $*$-algebras over~$X$ is the space $\Gamma(X,\mathcal{A})$
of all continuous sections of~$\mathcal{A}$ which, when equipped with the
usual pointwise defined operations of addition, scalar multiplication,
multiplication and involution, is easily seen to become a $*$-algebra.
Moreover, given a directed set $\Sigma$ of fiber seminorms $s$
on~$\mathcal{A}$ that generates its topology, as explained above,
we obtain a directed set of seminorms $\|.\|_{s,K}$ on~$\Gamma(X,
\mathcal{A})$ by taking the usual sup seminorms over compact
subsets $K$ of~$X$,
\begin{equation} \label{eq:SUPN1} 
 \|\varphi\|_{s,K}~=~\sup_{x \in K} s(\varphi(x))
 \qquad \mbox{for $\, \varphi \in \Gamma(X,\mathcal{A})$}~,
\end{equation}
turning $\Gamma(X,\mathcal{A})$ into a locally convex $*$-algebra
with respect to what we may continue to call the topology of uniform
convergence on compact subsets.
Over and above that, $\Gamma(X,\mathcal{A})$ carries two important
additional structures.
The first is that $\Gamma(X,\mathcal{A})$ is a \emph{module} over
the locally convex $*$-algebra $C(X)$ of continuous functions on~$X$,
as expressed by the compatibility conditions
\begin{equation} \label{eq:LCSTMOD} 
 \begin{array}{c}
  f (\varphi_1 \varphi_2)~=~(f \varphi_1) \, \varphi_2~
  =~\varphi_1 \, (f \varphi_2) \quad , \quad
  (f \varphi)^*~=~\bar{f} \varphi^* \\[2mm]
  \| f \varphi \|_{s,K}~\leqslant~\| f \|_K \| \varphi \|_{s,K} \\[1ex]
  \mbox{for $\, f \in C(X)$, $\varphi,\varphi_1,\varphi_2
                  \in \Gamma(X,\mathcal{A})$~.}
 \end{array}
\end{equation}
The second additional structure is that $\Gamma(X,\mathcal{A})$
comes equipped with a family $(\mathcal{\delta}_x)_{x \in X}$, indexed
by the points $x$ of the base space~$X$, of continuous homomorphisms
of locally convex $*$-algebras, the evaluation maps
\begin{equation} \label{eq:DIRDEL1} 
 \begin{array}{cccc}
  \delta_x : & \Gamma(X,\mathcal{A})
             & \longrightarrow & \mathcal{A}_x \\[1mm]
             & \varphi
             &   \longmapsto   &  \varphi(x)
 \end{array}~.
\end{equation}

Obviously, when $X$ is compact, we can omit the reference to compact
subsets since then $C(X)$ comes with the natural sup norm while every
fiber seminorm $s$ on~$\mathcal{A}$ will generate a seminorm $\|.\|_s$
on~$\Gamma(X,\mathcal{A})$ by taking the sup over all of~$X$; the
resulting topology is simply that of uniform convergence on all of~$X$.
On the other hand, when $X$ is locally compact but not compact, the
situation is a bit more complicated since we have to worry about the
behavior of functions and sections at infinity.
One way to deal with this issue consists in restricting to the algebras
$C_0(X)$ of continuous functions on~$X$ and $\Gamma_0(X,\mathcal{A})$
of continuous sections of~$\mathcal{A}$ that vanish at infinity
(in the usual sense that $\, f \in C(X) \,$ belongs to $C_0(X)$
and $\, \varphi \in \Gamma(X,\mathcal{A}) \,$ belongs to
$\Gamma_0(X,\mathcal{A})$ if for each $\, \epsilon > 0 \,$
and, in the second case, each $\, s \in \Sigma$, there exists
a compact subset $K$ of~$X$ such that $\, |f(x)| < \epsilon \,$
and $\, s(\varphi(x)) < \epsilon \,$ whenever $\, x \notin K$):
as in the compact case, these are locally convex $*$-algebras
with respect to the topology of uniform convergence on all of~%
$X$ and the latter is a module over the former, with the same
compatibility conditions and the same evaluation maps as before
(see equations~(\ref{eq:LCSTMOD}) and~(\ref{eq:DIRDEL1})).
Moreover, we have a condition of nondegeneracy, which is necessary
since we are now dealing with nonunital $*$-algebras: it states that
the ideal\,$^{\ref{fn:IDEAL}}$ generated by elements of the form $f \varphi$,
with $\, f \in C_0(X) \,$ and $\, \varphi \in \Gamma_0(X,\mathcal{A})$,
should be the entire algebra~$\Gamma_0(X,\mathcal{A})$.%
\footnote{Note that the condition of nondegeneracy can equally well
be formulated in the compact case but is trivially satisfied there
since the condition of vanishing at infinity is then void and so
we can identify $C_0(X)$ with $C(X)$, which has a unit, and
$\Gamma_0(X,\mathcal{A})$ with $\Gamma(X,\mathcal{A})$.}
An alternative choice would be to consider the (larger)
algebras $C_b(X)$ of bounded continuous functions on~$X$
and $\Gamma_b(X,\mathcal{A})$ of bounded continuous sections
of~$\mathcal{A}$,%
\footnote{A section of~$\mathcal{A}$ is said to be bounded if
its composition with each fiber seminorm $\, s \in \Sigma \,$
is bounded.}
again with the topology of uniform convergence on all of~$X$,
which has the advantage that $C_b(X)$ is unital.
In fact, both $C_0(X)$ and $C_b(X)$ are $C^*$-algebras, and the
latter is the multiplier algebra of the former:
\begin{equation}
 C_b(X)~=~M(C_0(X))~.
\end{equation}

All these constructions of section algebras become particularly useful
when we start out from an upper semicontinuous $C^*$-bundle $\mathcal{A}$
over~$X$.
In that case, $\Gamma_0(X,\mathcal{A})$ and $\Gamma_b(X,\mathcal{A})$
will both be $C^*$-algebras (which coincide among themselves and with
$\Gamma(X,\mathcal{A})$ when $X$ is compact), and the aforementioned
structure of $\Gamma_0(X,\mathcal{A})$ as a $C_0(X)$-module can be
reinterpreted as providing a $C^*$-algebra homomorphism $\, \Phi:
C_0(X) \longrightarrow Z(M(\Gamma_0(X,\mathcal{A})))$, where
$M(\Gamma_0(X,\mathcal{A}))$ is the multiplier algebra of~%
$\Gamma_0(X,\mathcal{A})$ and $Z(M(\Gamma_0(X,\mathcal{A})))$
its center.
Thus the section algebra $\Gamma_0(X,\mathcal{A})$ is a $C_0(X)$-algebra
in the sense of Kasparov~\cite{KP:EKKT}:

\begin{dfn} \label{def:CZERO}
 Given a locally compact topological space $X$, a $C_0(X)$-algebra
 is a $C^*$-algebra $A$ equipped with a $C^*$-algebra homomorphism
 \begin{equation} \label{eq:CZER1}
  \Phi: C_0(X) \longrightarrow Z(M(A))
 \end{equation}
 which is nondegenerate, i.e., such that the ideal generated by
 elements of the form $f a$, with $\, f \in C_0(X) \,$ and
 $\, a \in A$, is the entire algebra $A$.%
 \footnote{We shall simply write $fa$, instead of~$\Phi(f)(a)$,
 whenever convenient.}
\end{dfn}

\noindent
Note that the nondegeneracy condition imposed in Definition~\ref{def:CZERO}
above means that $\Phi$ extends uniquely to a $C^*$-algebra homomorphism
\begin{equation} \label{eq:CZER2}
 \Phi: C_b(X) \longrightarrow Z(M(A))
\end{equation}
i.e., $C_0(X)$-algebras are automatically also $C_b(X)$-algebras.
However, not every $C_b(X)$-algebras is also a $C_0(X)$-algebra,
since the nondegeneracy condition may fail: an obvious example is
provided by $C_b(X)$ itself, which is trivially a module over $C_b(X)$
itself and hence also over $C_0(X)$ but, as such, is degenerate; in
fact, in this case the ideal mentioned in Definition~\ref{def:CZERO}
above is $C_0(X)$ and not all of $C_b(X)$.
At any rate, in the context of the present paper, the extension
of the module structure from $C_0(X)$ to~$C_b(X)$ will not play
any significant role.

The notion of a $C_0(X)$-algebra homomorphism is, once again, the natural
one: it is a $*$-algebra homomorphism which is also a homomorphism of
$C_0(X)$-modules.

With these concepts at our disposal, we can now think of the process
of passing from bundles to their section algebras as a \emph{functor}.
More precisely, the version of interest here is the following: given
any locally compact topological space $X$, we have a corresponding
\emph{section algebra functor}
\begin{equation} 
 \Gamma_0(X,\cdot): \mathsf{C}_{\mathsf{us}}^*\mathsf{Bun}(X) \longrightarrow
                    \mathsf{C}_0(X)\mathsf{Alg}
\end{equation}
from the category $\mathsf{C}_{\mathsf{us}}^*\mathsf{Bun}(X)$ of upper
semicontinuous $C^*$-bundles over~$X$, whose morphisms are the strict
bundle morphisms over~$X$, to the category $\mathsf{C}_0(X)\mathsf{Alg}$
of~$C_0(X)$-algebras, whose morphisms are the $C_0(X)$-algebra homomorphisms.
Indeed, it is clear that given any strict bundle morphism $\, \phi: \mathcal{A}
\longrightarrow \mathcal{B} \,$ between upper semicontinuous $C^*$-bundles
$\mathcal{A}$ and~$\mathcal{B}$ over~$X$, pushing forward sections with~$\phi$
provides a corresponding $C_0(X)$-algebra homomorphism $\, \Gamma_0(X,\phi):
\Gamma_0(X,\mathcal{A}) \longrightarrow \Gamma_0(X,\mathcal{B})$.

Conversely, we can construct a \emph{sectional representation functor}
\begin{equation}
 \mathrm{SR}(X,\cdot): \mathsf{C}_0(X)\mathsf{Alg} \longrightarrow
                       \mathsf{C}_{\mathsf{us}}^*\mathsf{Bun}(X)
\end{equation}
as follows. First, given any $C_0(X)$-algebra~$A$, we define
$\mathrm{SR}(X,A)$, as a ``bundle'' over~$X$ (in the purely
set-theoretical sense), by writing
\begin{equation} \label{eq:CSTFIB1}
 \mathrm{SR}(X,A)~=~\bigcup_{x \in X}^{\scriptscriptstyle{\bullet}} \mathrm{SR}(X,A)_x
\end{equation}
where the fiber $\mathrm{SR}(X,A)_x$ over any point $x$ in~$X$ is defined by
\begin{equation} \label{eq:CSTFIB2}
 \mathrm{SR}(X,A)_x~=~A \: / \: \overline{\Phi(I_x) A}
 \quad \mbox{where} \quad I_x~=~\{ f \in C_0(X) \: | \: f(x)=0 \}~.
\end{equation}
The structure of~$\mathrm{SR}(X,A)$ as a $C^*$-bundle is then determined
by the construction \mbox{described} in Theorem~\ref{thm:FELL1} above,
specialized to the case of $C^*$-bundles and with
\begin{equation}
 \Gamma~=~\bigl\{ \varphi_a \, | \, a \in A \bigr\}
 \quad \mbox{where} \quad
 \varphi_a(x)  =  a \, + \, \overline{\Phi(I_x) A}~
              \in A \: / \: \overline{\Phi(I_x) A}~,
\end{equation}
since this space $\Gamma$ satisfies the two conditions of~Theorem~%
\ref{thm:FELL1} (condition~(b) is obvious and condition~(a) is shown
in~\cite[Proposition~C.10, p.~357]{Wil}).
Second, given any homo\-morphism $\, \phi_X: A \longrightarrow B \,$
between $C_0(X)$-algebras $A$ and~$B$, passing to quotients provides
a corresponding strict bundle morphism $\, \mathrm{SR}(X,\phi_X):
\mathrm{SR}(X,A) \longrightarrow \mathrm{SR}(X,B)$.

The construction outlined in the previous paragraph is actually the central
point in the proof of a famous theorem in the field, generally known as the
sectional representation theorem, which asserts that every $C_0(X)$-algebra
$A$ can be obtained as the section algebra $\Gamma_0(X,\mathcal{A})$ of an
appropriate upper semicontinuous $C^*$-bundle $\mathcal{A}$ over~$X$; for
an explicit statement with a complete proof, the reader is referred to~%
\cite[Theorem~C.26, p.~367]{Wil}.
Here, we state a strengthened version of this theorem, which extends it
to an equivalence of categories~\cite[p.~18]{ML}:
\begin{thm}[Sectional Representation Theorem] \label{thm:SecRep}
 Given a locally compact topological space\/~$X$, the functors
 $\Gamma_0(X,\cdot)$ and $\mathrm{SR}(X,\cdot)$ establish an
 equivalence between the categories $\mathsf{C}_{\mathsf{us}}^*%
 \mathsf{Bun}(X)$ and $\mathsf{C}_0(X)\mathsf{Alg}$.
\end{thm}
\begin{proof}
 Explicitly, the statement of the theorem means that, for any upper
 semicontinuous $C^*$-bundle~$\mathcal{A}$ over~$X$, there is a
 strict bundle isomorphism $\, \mathcal{A} \cong \mathrm{SR}
 \bigl( X , \Gamma_0(X,\mathcal{A}) \bigr) \,$ which behaves
 naturally under strict bundle morphisms, and similarly that, for
 any $C_0(X)$-algebra~$A$, there is a $C_0(X)$-algebra isomorphism
 $\, A \cong \Gamma_0 \bigl( X , \mathrm{SR}(X,A) \bigr) \,$ which
 behaves naturally under $C_0(X)$-algebra homomorphisms.
 The existence of the second of these isomorphisms is precisely
 the content of the traditional formulation of the sectional
 representation theorem~\cite[Theorem~C.26, p.~367]{Wil},
 whereas the first is constructed similarly.
 Namely, given any upper semicontinuous $C^*$-bundle~$\mathcal{A}$
 over~$X$, note that, for any point $x$ in~$X$, we have
 \[
  \overline{\Phi(I_x) \, \Gamma_0(X,\mathcal{A})}~
  =~\bigl\{ \varphi \in \Gamma_0(X,\mathcal{A}) \, | \,
            \varphi(x) = 0 \bigr\}
 \]
 since the inclusion $\subset$ is trivial and the inclusion $\supset$
 follows from a standard argument: given $\, \varphi \in \Gamma_0%
 (X,\mathcal{A}) \,$ and any $\, \epsilon > 0$, there are an open
 neighborhood $U$ of~$x$ with compact closure $\bar{U}$ and a
 compact subset $K$ containing it such that the function
 $\, x \longmapsto \| \varphi(x) \|_x$ \linebreak is
 $< \epsilon$ in~$U$ (since it vanishes at~$x$ and the $C^*$
 fiber norm on~$\mathcal{A}$ is upper semicontinuous) as well
 as outside of~$K$ (since $\varphi$ vanishes at infinity), so
 applying Urysohn's lemma we can find a function $\, f \in
 C_c(X) \,$ with $\, 0 \leqslant f \leqslant 1 \,$ which is
 $\equiv 0$ outside of~$U$ but satisfies $\, f(x) = 1 \,$
 and combine it with another function $\, g \in C_0(X) \,$
 with $\, 0 \leqslant f \leqslant 1 \,$ which is $\equiv 1$
 on~$K$ to get a function $\, (1-f) g \in C_0(X) \,$ which
 is $\equiv 1$ on $K \setminus U$ but vanishes at~$x$ and
 from that deduce that the sup norm of $\, \varphi - (1-f)
 g \varphi \,$ is $< \epsilon$. Therefore, for any point~$x$
 in~$X$, we get a $C^*$-algebra isomorphism
 \[
  \mathrm{SR} \bigl( X, \Gamma_0(X,\mathcal{A}) \bigr)_x
  \cong \mathcal{A}_x
 \]
 which provides the desired bundle isomorphism as $x$ varies
 over the base space~$X$.
\end{proof}

An interesting question in this context would be to fully incorporate
the notions of pull-back and of change of base ring into this picture.
On the one hand, given any proper continuous map
$\, f: X \longrightarrow Y \,$ between locally
compact topological spaces $X$ and~$Y$, we can
define a corresponding \emph{pull-back functor}
\begin{equation}
 f^*: \mathsf{C}_{\mathsf{us}}^*\mathsf{Bun}(Y) \longrightarrow
      \mathsf{C}_{\mathsf{us}}^*\mathsf{Bun}(X)~,
\end{equation}
associating to each upper semicontinuous $C^*$-bundle $\mathcal{B}$
over~$Y$ its pull-back via~$f$, which is an upper semicontinuous $C^*$-%
bundle~$f^* \mathcal{B}$ over~$X$, fiberwise defined by $\, (f^* \mathcal{B})_x
= \mathcal{B}_{f(x)}$, and associating to each strict bundle morphism $\, \phi:
\mathcal{B} \longrightarrow \mathcal{B}' \,$ over~$Y$ its pull-back via~$f$,
which is a strict bundle morphism $\, f^* \phi: f^* \mathcal{B}
\longrightarrow f^* \mathcal{B}' \,$ over~$X$, fiberwise defined by
$\, f^* \phi \big|_{(f^* \mathcal{B})_x} = \phi \big|_{\mathcal{B}_{f(x)}}$.
On the other hand, given any proper continuous map
$\, f: X \longrightarrow Y$ \linebreak between locally compact
topological spaces $X$ and~$Y$, we can define a corresponding
\emph{change of base ring functor}
\begin{equation}
 f_\sharp: \mathsf{C}_0(X)\mathsf{Alg} \longrightarrow
          \mathsf{C}_0(Y)\mathsf{Alg}~,
\end{equation}
associating to each $C_0(X)$-algebra~$A$ a $C_0(Y)$-algebra $f_\sharp A$
which as a $C^*$-algebra is equal to~$A$ but with a modified module
structure, defining multiplication with functions in~$C_0(Y)$ to be
given by multiplication with the corresponding functions in~$C_0(X)$
obtained by pull-back via $f$, and associating to each $C_0(X)$-algebra
homomorphism $\, \phi_X: A \longrightarrow A'$ \linebreak a $C_0(Y)$-algebra
homomorphism $\, f_\sharp \phi_X: f_\sharp A \longrightarrow f_\sharp A' \,$
which as a $C^*$-algebra homo\-morphism is equal to $\phi_X$ but is now
linear with respect to the modified module structure.
It should be noted that these two functors do \emph{not} translate
into each other under the equivalence established by the sectional
representation theorem because, roughly speaking, they go in opposite
directions and the first preserves the fibers while changing the section
algebras whereas the second preserves the section algebras while changing
the fibers. \linebreak
Indeed, for any upper semicontinuous $C^*$-bundle $\mathcal{B}$ over~$Y$,
composition of sections with $f$ induces a $C^*$-algebra homomorphism
\begin{equation}
 f^*: \Gamma_0(Y,\mathcal{B})~\longrightarrow~\Gamma_0(X,f^* \mathcal{B})
\end{equation}
which, in general, is far from being an isomorphism since it may have
a nontrivial kernel (consisting of sections of~$\mathcal{B}$ over~$Y$
that vanish on the image of~$f$) as well as a non\-trivial image (consisting
of sections of~$f^* \mathcal{B}$ over~$X$ that are constant along the level
sets of~$f$).
Similarly, given any $C_0(X)$-algebra~$A$ and using~$f$ to also consider it
as a $C_0(Y)$-algebra~$f_\sharp A$, we can apply the respective sectional
representation functors to introduce the corresponding $C^*$-bundles
$\, \mathcal{A} = \mathrm{SR}(X,A) \,$ over~$X$ and $\, f_\sharp \mathcal{A}
= \mathrm{SR}(Y,f_\sharp A) \,$ over~$Y$, so that $\, A \cong \Gamma_0%
(X,\mathcal{A})\,$ and $\, f_\sharp A \cong \Gamma_0(Y,f_\sharp \mathcal{A})$:
then we find that the fibers of $f_\sharp \mathcal{A}$ are related to the
fibers of $\mathcal{A}$ by
\begin{equation} \label{eq:fcb}
 (f_\sharp \mathcal{A})_y~\cong~\Gamma(f^{-1}(y),\mathcal{A})~.
\end{equation}

We are now ready to address the central point of this section, namely,
the construction of the $C^*$-completion at the level of bundles and
its relation with the $C^*$-completion at the level of the corresponding
section algebras.
To this end, suppose that $X$ is a locally compact topological space and
$\mathcal{A}$ is a bundle of locally convex $*$-algebras over~$X$, with
bundle projection $\, \rho: \mathcal{A} \longrightarrow X \,$ and with
respect to some directed set~$\Sigma$ of fiber seminorms on~$\mathcal{A}$,
as in Definition~\ref{def:BUNALG} above.
Suppose furthermore that we are given a function \mbox{$\|.\|: \mathcal{A}
\longrightarrow \mathbb{R}$}\, which is a $C^*$ fiber norm (in the sense
of inducing a $C^*$-norm on each fiber of~$\mathcal{A}$).
From these data, we can construct a ``fiberwise $C^*$-completion''
$\bar{\mathcal{A}}$ of~$\mathcal{A}$ by taking, for every point $x$
of~$X$, the completion $\bar{\mathcal{A}}_x$ of~$\mathcal{A}_x$
with respect to the $C^*$-norm $\|.\|_x$ to obtain a family
$(\bar{\mathcal{A}}_x)_{x \in X}$ of $C^*$-algebras and consider
the disjoint union
\begin{equation} \label{eq:CSTFIB3}
 \bar{\mathcal{A}}~=
 \bigcup_{x \in X}^{\scriptscriptstyle{\bullet}} \bar{\mathcal{A}}_x
\end{equation}
as a ``bundle'' of $C^*$-algebras over~$X$ (in the purely set-theoretical
sense); obviously, $\mathcal{A} \subset \bar{\mathcal{A}}$ and the original
bundle projection $\, \rho: \mathcal{A} \longrightarrow X \,$ is simply the
restriction of the bundle projection $\, \bar{\rho}: \bar{\mathcal{A}}
\longrightarrow X$.
In order to control the topological aspects involved in this
construction, we have to impose additional hypotheses.
Namely, we shall assume that $\mathcal{A}$ is an upper semicontinuous
bundle of locally convex $*$-algebras over~$X$, as in Definition~%
\ref{def:BUNALG} above, and that $\|.\|$ is \emph{locally bounded}
by~$\Sigma$, i.e., for every point $x$ of~$X$ there exist a
neighborhood $U_x$ of~$x$, a fiber seminorm $s$ belonging
to~$\Sigma$ and a constant $\, C > 0 \,$ such that $\, \|a\|
\leqslant C \, s(a)$ \linebreak for $\, a \in \rho^{-1}(U_x)$.
Our goal will be to show that, under these circumstances, $\bar{\mathcal{A}}$
admits a unique topology turning it into an upper semicontinuous $C^*$-bundle
over~$X$ such that the space of its continuous sections vanishing at infinity
is the completion of the space of continuous sections of compact support of~%
$\mathcal{A}$ with respect to the sup norm induced by the $C^*$ fiber norm~%
$\|.\|$: this will provide us with a natural and concrete example of the
abstract sectional representation theorem (Theorem~\ref{thm:SecRep}).

To do so, note first that since the fiber norm $\|.\|$ is locally bounded by
the seminorms in~$\Sigma$ which are upper semicontinuous, and since $X$ is
locally compact, it follows that, given any continuous section $\varphi$
of~$\mathcal{A}$, the function $\, x \longmapsto \| \varphi(x) \|_x \,$
on~$X$ is locally bounded
and hence bounded on compact subsets of~$X$, so for each compactly
supported continuous section $\varphi$ of~$\mathcal{A}$,
$\|\varphi\|_\infty = \sup_{x \in X} \|\varphi(x)\|_x \,$ exists.
It is then clear that $\|.\|_\infty$ defines a $C^*$-norm on
$\Gamma_c(X,\mathcal{A})$: let $\overline{\Gamma_c(X,\mathcal{A})}$
be the corresponding $C^*$-completion.
Next, note that $\Gamma_c(X,\mathcal{A})$ is also a module over
$C_0(X)$, and hence so is $\overline{\Gamma_c(X,\mathcal{A})}$
(since multiplication is obviously a continuous bilinear map with
respect to the pertinent $C^*$-norms); moreover, we have the equality
$\, C_0(X) \, \Gamma_c(X,\mathcal{A}) = \Gamma_c(X,\mathcal{A})$, since
any $\, \varphi \in \Gamma_c(X,\mathcal{A}) \,$ can be written in the
form $f \varphi$ for some $\, f \in C_0(X)$ (it suffices to choose $f$
to be equal to~$1$ on the support of~$\varphi$, using Urysohn's lemma),
so $\overline{\Gamma_c(X,\mathcal{A})}$ is in fact a $C_0(X)$-algebra.
Therefore, by the construction of the sectional representation functor,
we have
\[
 \bar{\mathcal{A}}~=~\mathrm{SR}(X,\overline{\Gamma_c(X,\mathcal{A})})~.
\]
In particular, $\bar{\mathcal{A}}$ admits a unique topology turning
it into an upper semicontinuous $C^*$-bundle over~$X$ such that
continuous sections of compact support of~$\mathcal{A}$ become
continuous sections of compact support of~$\bar{\mathcal{A}}$,
since the space $\Gamma_c(X,\mathcal{A})$ satisfies the two
conditions of~Theorem~\ref{thm:FELL1} (condition~(b) is obvious
and condition~(a) is stated in~\cite[Proposition~C.10, p.~357]{Wil}).
In~fact, it follows from the construction of the topology on~%
$\bar{\mathcal{A}}$ in the proof of Theorem~\ref{thm:FELL1} that the
inclusion $\, \mathcal{A} \subset \bar{\mathcal{A}} \,$ is continuous,
since if $U$ is a sufficiently small open subset of~$X$ such that
$\, \|a\| \leqslant C \, s(a) \,$ for $\, a \in \rho^{-1}(U) \,$
with some fiber seminorm $\, s \in \Sigma \,$ and some constant
$\, C>0$, we have, for any $\varphi \in \Gamma_c(X,\mathcal{A})$,
\[
 W_{\mathcal{A}}(\varphi,U,s,\epsilon/C)~
 \subset~W_{\bar{\mathcal{A}}}(\varphi,U,\|.\|,\epsilon) \cap \mathcal{A}~,
\]
and $W_{\mathcal{A}}(\varphi,U,s,\epsilon/C)$ is open in~$\mathcal{A}$
since $\varphi$ is continuous and $s$ is upper semicontinuous; this
fact also implies that the original $C^*$ fiber norm $\|.\|$ on~%
$\mathcal{A}$, just like its extension to the $C^*$ fiber norm
(also denoted by~$\|.\|$) on~$\bar{\mathcal{A}}$, is automatically
upper semicontinuous; moreover, $\mathcal{A}$ is dense in~$\bar%
{\mathcal{A}}$ (simply because, by construction, every fiber
$\mathcal{A}_x$ of~$\mathcal{A}$ is dense in the corresponding
fiber $\bar{\mathcal{A}}_x$ of~$\bar{\mathcal{A}}$).
All of this justifies calling $\bar{\mathcal{A}}$ the \emph{fiberwise
$C^*$-completion} of~$\mathcal{A}$ with respect to the given $C^*$
fiber norm.
And finally, it is clear that, by construction, the section algebra
$\Gamma_0(X,\bar{\mathcal{A}})$ is the $C^*$-completion of the section
algebra~$\Gamma_c(X,\mathcal{A})$ with respect to the sup norm
$\|.\|_\infty$:
\begin{thm}[Bundle Completion Theorem] \label{thm:BunCom}
 Given a locally compact topological space\/~$X$, let\/ $\mathcal{A}$ be
 an upper semicontinuous bundle of locally convex $*$-algebras over\/~$X$,
 with respect to some directed set\/~$\Sigma$ of fiber seminorms, let\/
 $\|.\|$ be a\/ $C^*$ fiber norm on\/~$\mathcal{A}$ which is locally
 bounded with respect to\/~$\Sigma$ and let\/ $\bar{\mathcal{A}}$ be
 the corresponding fiberwise\/ $C^*$-completion of\/~$\mathcal{A}$.
 Then there is a unique topology  on\/~$\bar{\mathcal{A}}$ turning it
 into an upper semicontinuous\/ $C^*$-bundle over\/~$X$ such that the
 $C^*$-completion of the section algebra\/ $\Gamma_c(X,\mathcal{A})$
 with respect to the sup norm\/ $\|.\|_\infty$ is the section algebra\/
 $\Gamma_0(X,\bar{\mathcal{A})}$:
 \begin{equation} \label{eq:BunCom}
  \overline{\Gamma_c(X,\mathcal{A})}~=~\Gamma_0(X,\bar{\mathcal{A}})~.
 \end{equation}
\end{thm}

Regarding universal properties of such $C^*$-completions at the level of
bundles and of their section algebras, it is now easy to see that these
depend essentially on whether the corresponding universal properties hold
fiberwise, at the level of algebras, provided we take into account that
when dealing with the section algebras, we must work in the category of
$C_0(X)$-algebras rather than just $C^*$-algebras.
More specifically, under the same hypotheses as before (namely that
$\mathcal{A}$ is is an upper semicontinuous bundle of locally convex
$*$-algebras over~$X$ and $\|.\|$ is a locally bounded $C^*$ fiber
norm on~$\mathcal{A}$), we can guarantee that:
\begin{itemize}
 \item Universality implies universality. If, for every point $x$
       of~$X$, $\|.\|_x$ is the maximal $C^*$-norm on~$\mathcal{A}_x$,
       then $\|.\|$ is the maximal $C^*$ fiber norm on~$\mathcal{A}$
       and we can refer to $\bar{\mathcal{A}}$ as the \emph{minimal
       $C^*$-completion} or \emph{universal enveloping $C^*$-bundle}
       of~$\mathcal{A}$. Moreover, $\Gamma_0(X,\bar{\mathcal{A}})$ will
       under these circumstances be the \emph{universal enveloping
       $C_0(X)$-algebra} of~$\Gamma_c(X,\mathcal{A})$.
 \item Uniqueness implies uniqueness. If, for every point $x$ of~$X$,
       $\mathcal{A}_x$ admits a unique $C^*$-norm and hence a unique
       $C^*$-algebra completion, then $\mathcal{A}$ admits a unique
       $C^*$ fiber norm and hence a unique $C^*$-bundle completion.
       Moreover, $\Gamma_0(X,\bar{\mathcal{A}})$ will under these
       circumstances be the unique $C_0(X)$-completion of~%
       $\Gamma_c(X,\mathcal{A})$.
\end{itemize}

\section{The DFR-Algebra for Poisson Vector Bundles}
\label{sec:DFRPois}

Let $(E,\sigma)$ be a Poisson vector bundle with base manifold~$X$, i.e., $E$
is a (smooth) real vector bundle of fiber dimension~$n$, say, over a (smooth)
manifold~$X$, with typical fiber~$\mathbb{E}$, equipped with a fixed (smooth)
bivector field $\sigma$; in other words, the dual $E^*$ of $E$ is a (smooth)
presymplectic vector bundle.%
\footnote{Note that we do \emph{not} require $\sigma$ to be nondegenerate
or even to have constant rank.}
Then it is clear that we can apply all the constructions of Section~3
to each fiber.
The question to be addressed in this section is how, using the methods
outlined in Section~4, the results can be glued together along the base
manifold~$X$ and to describe the resulting global objects.

Starting with the collection of Heisenberg algebras $\mathfrak{h}_{\sigma(x)}$
($x \in X$), we note first of all that these fit together into a (smooth) real
vector bundle over~$X$, which is just the direct sum of~$E^*$ and the trivial
line bundle $\, X \times \mathbb{R} \,$ over~$X$.
The nontrivial part is the commutator, which is defined by equation~%
(\ref{eq:HACOMM}) applied to each fiber, turning this vector bundle into
a \emph{totally intransitive Lie algebroid} \cite[Def.~3.3.1, p.~100]{MK}
which we shall call the \emph{Heisenberg algebroid} associated to~%
$(E,\sigma)$ and denote by $\mathfrak{h}(E,\sigma)$: it will even be
a \emph{Lie algebra bundle} \cite[Def.~3.3.8, p.~104]{MK} if and only
if $\sigma$ has constant rank. \linebreak
Of course, spaces of sections (with certain regularity properties)
of~$\mathfrak{h}(E,\sigma)$ will then form (infinite-dimensional)
Lie algebras with respect to the (pointwise defined) commutator,
but the correct choice of regularity conditions is a question of
functional analytic nature to be dictated by the problem at hand.

Similarly, considering the collection of Heisenberg groups~$H_{\sigma(x)}$
($x \in X$), we note that these fit together into a (smooth) real fiber
bundle over~$X$, which is just the fiber product of~$E^*$ and the trivial
line bundle $X \times \mathbb{R}$.
Again, the nontrivial part is the product, which is defined by equation~%
(\ref{eq:HGPROD}) applied to each fiber, turning this fiber bundle into
a \emph{totally intransitive Lie groupoid} \cite[Def.~1.1.3, p.~5 \&
Def.~1.5.9, p.~32]{MK} which we shall call the \emph{Heisenberg groupoid}
associated to~$(E,\sigma)$ and denote by $H(E,\sigma)$: it will even be a
\emph{Lie group bundle} \cite[Def.~1.1.19, p.~11]{MK} if and only if
$\sigma$ has constant rank.
And again, spaces of sections (with certain regularity properties)
of~$H(E,\sigma)$ will form (infinite-dimensional) Lie groups with
respect to the (pointwise defined) product, but the correct choice
of regularity conditions is a question of functional analytic nature
to be dictated by the problem at hand.

An analogous strategy can be applied to the collection of Heisenberg
$C^*$-algebras $\mathcal{E}_{\sigma(x)}$ and $\mathcal{H}_{\sigma(x)}$
($x \in X$), but the details are somewhat intricate since the fibers
are now (infinite-dimensional) $C^*$-algebras which may depend on the
base point in a discontinuous way, since the rank of~$\sigma$ is
allowed to jump.
Still, there remains the question whether we can fit the collections of
Heisenberg $C^*$-algebras $\mathcal{E}_{\sigma(x)}$ and $\mathcal{H}_{\sigma(x)}$
into $C^*$-bundles over~$X$ which are at least upper semicontinuous.


The basic idea that allows us to bypass all these difficulties is to
introduce two smooth vector bundles over~$X$, denoted in what follows
by $\mathcal{S}(E)$ and by $\mathcal{B}(E)$, whose fibers are just
the Fr\'echet spaces of Schwartz functions and of totally bounded
smooth functions on the fibers of the original vector bundle~$E$,
respectively, i.e., $\mathcal{S}(E)_x = \mathcal{S}(E_x)$ and
\mbox{$\, \mathcal{B}(E)_x = \mathcal{B}(E_x)$:} note that
choosing any system of local trivializations of the original
vector bundle~$E$ will give rise to induced systems of local
trivializations which, together with an adequate partition of
unity, can be used to provide appropriate systems of fiber
seminorms, both for~$\mathcal{S}(E)$ and for~$\mathcal{B}(E)$.
Moreover, we use the Poisson bivector field~$\sigma$ to introduce
a fiberwise Weyl-Moyal star product on these vector bundles which,
when combined with the standard fiberwise involution, will turn
them into continuous bundles of Fr\'echet $*$-algebras,%
\footnote{Continuity of the Weyl-Moyal star product follows from
the estimate of Proposition~\ref{prp:SCHWEST} in the Appendix, in
the case of~$\mathcal{S}$, and from~\cite[Proposition~2.2, p.~12]%
{RDQ}, in the case of~$\mathcal{B}$.}
here denoted by $\mathcal{S}(E,\sigma)$ and by $\mathcal{B}(E,\sigma)$,
respectively.
We stress that even though both are locally trivial (and smooth)
as vector bundles over~$X$, they will fail to be locally trivial as
Fr\'echet $*$-algebra bundles~-- unless $\sigma$ has constant rank.%
\footnote{This is exactly the same situation as for the fiberwise
commutator in the Heisenberg algebroid or the fiberwise product
in the Heisenberg grupoid.}

The next step consists in gathering the $C^*$-norms on the fibers
of these two bundles, as defined in Section~\ref{sec:Halg}, to
construct $C^*$-fiber norms on each of them which, due to the
estimate~(\ref{eq:LREGR5}), are locally bounded. Therefore, as
seen in Section~4, they admit $C^*$-completions which we call the
\emph{DFR-bundles}, here denoted by $\mathcal{E}(E,\sigma)$ and by
$\mathcal{H}(E,\sigma)$, respectively; thus
\begin{equation}
 \mathcal{E}(E,\sigma)~=~\overline{\mathcal{S}(E,\sigma)}~~~,~~~
 \mathcal{H}(E,\sigma)~=~\overline{\mathcal{B}(E,\sigma)}~.
\end{equation}
Their section algebras are then called the \emph{DFR-algebras}.

We stress that this is a canonical construction because the Heisenberg
$C^*$-algebras are the universal enveloping $C^*$-algebras associated to
the Heisenberg-Schwartz and Heisenberg-Rieffel algebras, and even more
than that, they are their \emph{only} $C^*$-completions, so that
according to the results obtained in Section~4, the same goes for
the corresponding bundles and section algebras: the DFR-bundles
are the universal enveloping $C^*$-bundles of the corresponding
Fr\'echet $*$-algebra bundles introduced above, and even more than
that, they are their \emph{only} $C^*$-completions, and an analogous
statement holds for the DFR-algebras as ``the'' $C^*$-completions of
the corresponding section algebras.

Of course, when $\sigma$ is nondegenerate, all these constructions can
be drastically simplified; in particular, the DFR-bundles $\mathcal{E}%
(E,\sigma)$ and $\mathcal{H}(E,\sigma)$ can be obtained much more
directly from the principal bundle of symplectic frames for~$E$ as
associated bundles, and the former becomes identical with the Weyl
bundle as constructed in~\cite{Ply}.


\subsection{Recovering the Original DFR-Model}
\label{sec:RODFRM}

An important special case of the general construction outlined in the
previous section occurs when the underlying manifold~$X$ and Poisson
vector bundle~$(E,\sigma)$ are homogeneous.
More specifically, assume that $G$ is a Lie group which acts properly
on~$X$ as well as on~$E$ and such that $\sigma$ is $G$-invariant: this
means that writing
\begin{equation}
 \begin{array}{ccc}
  G \times X & \longrightarrow &     X     \\[1mm]
     (g,x)   &   \longmapsto   & g \cdot x
 \end{array}
 \qquad \mathrm{and} \qquad
 \begin{array}{ccc}
  G \times E & \longrightarrow &     E     \\[1mm]
     (g,u)   &   \longmapsto   & g \cdot u
 \end{array}
\end{equation}
for the respective actions, where the latter is linear along the fibers
and hence induces an action
\begin{equation}
 \begin{array}{ccc}
  G \times \bwedge^2 E & \longrightarrow & \bwedge^2 E \\[1mm]
          (g,u)        &   \longmapsto   &  g \cdot u
 \end{array}
\end{equation}
we should have
\begin{equation}
 \sigma(g \cdot x)~=~g \cdot \sigma(x)
 \qquad \mathrm{for}~~g \in G \,,\, x \in X~.
\end{equation}
Moreover, we shall assume that the action of~$G$ on the base manifold~$X$
is transitive.
Then, choosing a reference point $x_0$ in~$X$ and denoting by $H$ its
stability group in~$G$, by $\mathbb{E}$ the fiber of~$E$ over $x_0$ and
by $\sigma_0$ the value of the bivector field $\sigma$ at~$x_0$, we can
identify: $X$ with the homogeneous space $G/H$, $E$ with the vector
bundle $\, G \times_H \mathbb{E} \,$ associated to~$G$ (viewed as
a principal $H$-bundle over~$G/H$) and to the representation of~$H$
on~$\mathbb{E}$ obtained from the action of~$G$ on~$E$ by appropriate
restriction, and $\sigma$ with the bivector field obtained from
$\sigma_0$ by the association process.
Explicitly, for example, we identify the left coset $\, gH \in G/H \,$
with the point $\, g \cdot x_0 \in X \,$ and, for any $\, u_0 \in \mathbb{E}$,
the equivalence class~$\, [g,u_0] = [gh,h^{-1} \cdot u_0] \in G \times_H
\mathbb{E} \,$ with the vector $\, g \cdot u_0 \in E$.
As a result, we see that if the representation of~$H$ on~$\mathbb{E}$
extends to a representation of~$G$, then the associated bundle 
$\, G \times_H \mathbb{E} \,$ is globally trivial: an explicit
trivialization is given by
\begin{equation}
 \begin{array}{ccc}
       G \times_H \mathbb{E}     & \longrightarrow & G/H \times \mathbb{E}
  \\[1mm]
  [g,u_0] = [gh,h^{-1} \cdot u_0] &   \longmapsto   &  (gH,g^{-1} \cdot u_0)
 \end{array}
\end{equation}

Of course, $G$-invariance combined with transitivity implies that $\sigma$
has constant rank and hence the Heisenberg algebroid becomes a Lie algebra
bundle, the Heisenberg groupoid becomes a Lie group bundle and the DFR-%
bundles $\mathcal{E}(E,\sigma)$ and $\mathcal{H}(E,\sigma)$ become
locally trivial (and smooth) $C^*$-bundles.
Moroever, if the representation of~$H$ on~$\mathbb{E}$ extends to
a representation of~$G$, all these bundles will even be globally
trivial.

To recover the original DFR-model, consider four-dimensional Minkowski
space $\mathbb{R}^{1,3}$, which has the Lorentz group $O(1,3)$ as its
isometry group, and choose any symplectic form on $\mathbb{R}^{1,3}$,
say the one defined by the matrix $\bigl( \begin{smallmatrix} 0 & 1_2 \\
-1_2 & 0 \end{smallmatrix} \bigr)$. Let $\sigma_0$ be the corresponding
Poisson tensor and $H$ be its stability group under the action of $O(1,3)$.
Then we may recover the space $\Sigma$ from the original paper~\cite{DFR}
as the quotient space~$O(1,3)/H$.
Moreover, the vector bundle $\, O(1,3) \times_H \mathbb{R}^{1,3} \,$
associated to the canonical principal $H$-bundle $O(1,3)$ over~$\Sigma$
and the defining representation of $\, H \subset O(1,3) \,$ on~%
$\mathbb{R}^{1,3}$ carries a canonical Poisson structure defined by
using the action of~$O(1,3)$ to transport the Poisson tensor $\sigma_0$
at the distinguished point $o \in \Sigma \,$ (i.e., $[1] \in O(1,3)/H$)
to all other points of~$\Sigma$. 
According to the previous discussion, the resulting DFR-bundles will be
globally trivial, and so we have
\begin{equation}\label{eqn:ODFRB}
 \mathcal{E} \bigl( O(1,3) \times_H \mathbb{R}^{1,3},\sigma \bigr)~
 \cong~\Sigma \times \mathcal{K}
 \quad \text{and} \quad
 \mathcal{H} \bigl( O(1,3) \times_H \mathbb{R}^{1,3},\sigma \bigr)~
 \cong~\Sigma \times \mathcal{B}~.
\end{equation}
Moreover, the corresponding DFR-algebra
\[
 \Gamma_0 \left(
 \mathcal{E} \bigl( O(1,3) \times_H \mathbb{R}^{1,3},\sigma \bigr) \right)
\]
will then be the same as the one originally defined in~\cite{DFR}.

\subsection{The DFR-Model Extended}

We can extend the construction above to obtain a $C^*$-bundle over
an arbitrary spacetime whose fibers are isomorphic to the original
DFR-algebra.
Let $(M,g)$ be an $n$-dimensional Lorentz manifold with orthonormal
frame bundle $O(M,g)$.
Also, let $\sigma_0$ be a fixed bivector on~$\mathbb{R}^n$ and
$\Sigma$ its orbit under the action of the Lorentz group $O(1,n-1)$.
Consider the associated fiber bundle
\[
 \Sigma(M)~=~O(M,g) \times_{O(1,n-1)} \Sigma
\] 
over~$M$, whose bundle projection we shall denote by~$\pi$.
Using $\pi$ to pull back the tangent bundle $TM$ of~$M$ to
$\Sigma(M)$, we obtain a vector bundle $\pi^* TM$ over $\Sigma(M)$
which carries a canonical bivector field $\sigma$ defined by the
original bivector $\sigma_0$. Then the section algebra
\[
  \Gamma_0 \left( \Sigma(M),\mathcal{E} \bigl( \pi^* TM,\sigma \bigr) \right)
\]
of the resulting DFR-bundle $\mathcal{E}(\pi^* TM,\sigma)$
is not only a $C_0(\Sigma(M))$-algebra but, using the bundle
projection $\pi$, it also becomes a $C_0(M)$-algebra and hence
can be regarded as the section algebra of a $C^*$-bundle over~$M$.
Refining the discussion in Section~4 (see, in particular, equation~%
(\ref{eq:fcb})), we can show that the fibers of this sectional
representation bundle are precisely the ``tangent space
DFR-algebras''
\begin{equation}\label{eqn:LCQSTF}
 \Gamma_0 \left( \Sigma(M)_m , \mathcal{E} \bigl(
                 O(T_m M,g_m) \times_{H_m} T_m M,\sigma_m \bigr) \right) .
\end{equation}
In analogy with the term ``quantum spacetime'' employed by the
authors of~\cite{DFR} to designate the original DFR-algebra,
we suggest to refer to the functor that to each Lorentz manifold
$(M,g)$ associates the section algebra $\Gamma_0(\Sigma(M),
\mathcal{E}(\pi^* TM,\sigma))$ as ``locally covariant quantum
spacetime''.

\section{Conclusions and Outlook}

Our first goal when starting this investigation was to find an appropriate
mathematical setting for geometrical generalizations of the DFR-model~-- a
model for ``quantum spacetime'' which grew out of the attempt to avoid
the conflict between the classical idea of sharp localization of events
(ideally, at points of spacetime) and the creation of black hole regions
and horizons by the concentration of energy and momentum needed to achieve
such a sharp localization, according to the Heisenberg uncertainty relations.
To begin with, \linebreak this required translating the Heisenberg uncertainty
relations into the realm of $C^*$-algebra theory in such a way as to maintain
complete control over the dependence on the under\-lying (pre)symplectic form:
a problem that we found can be completely solved within Rieffel's theory of
strict deformation quantization, leading to a new construction of ``the
$C^*$-algebra of the canonical commutation relations'' which is an
alternative to existing ones such as the Weyl algebra~\cite{MA,MSTV}
or the resolvent algebra~\cite{BG1}.
The other main ingredient that had to be incorporated and further
developed was the general theory of bundles of locally convex $*$-%
algebras and, in particular, how the process of $C^*$-completion of
$*$-algebras at the level of fibers relates to that at the level
of section algebras.
The main outcome here is the definition of a novel procedure of
$C^*$-completion, now at the level of bundles, which to each bundle
of locally convex $*$-algebras, equipped with a locally bounded $C^*$
fiber seminorm, associates a $C^*$-bundle over the same base space
such that, at the level of $*$-algebras, the fibers of the latter
are the $C^*$-completions of the fibers of the former and, with
appropriate falloff conditions at infinity, the $C^*$-algebra of
continuous sections of the latter is the $C^*$-completion of the
$*$-algebra of continuous sections of the former.
Combining these two ingredients, we arrive at a generalization of
the mathematical construction underlying the DFR-model which, among
other things, can be applied in any dimension and in curved spacetime.

It should perhaps be emphasized at this point that it is not clear how
much of the original physical motivation behind the DFR-model carries
over to our mathematical generalization.
However, we believe our construction to be of interest in its own right,
as a tool to generate a nontrivial class of $C^*$-bundles (the DFR-%
bundles), each of which can be obtained as the (in this case, unique) 
$C^*$-completion of a concrete bundle of Fr\'echet $*$-algebras that
is canonically constructed from a given finite-dimensional Poisson vector 
bundle and, as a bundle of Fr\'echet spaces, is locally trivial and even 
smooth.
This whole process can be generalized even further by considering
other methods to generate $C^*$-algebras from an appropriate class of 
vector spaces (to replace the passage from pre-symplectic vector spaces 
to Heisenberg $C^*$-algebras) which satisfy continuity conditions in such 
a way as to allow for a lift from vector spaces and $C^*$-algebras to vector 
bundles and $C^*$-bundles, in the spirit of the functor lifting theorem 
\cite{LA}.
  
The construction of the aforementioned bundle of Fr\'echet $*$-algebras
gains additional importance when one considers the necessity of identifying 
further geometrical structures on the ``noncommutative spaces'' that the 
DFR-algebras are supposed to emulate.
A~first step in this direction is to look at the general definition of smooth
subalgebras of $C^*$-algebras, as discussed in \cite{BC} and \cite{BIO}. 
Using the results from the previous sections, it is a trivial exercise to
show that  the Heisenberg-Schwartz and Heisenberg-Rieffel algebras are smooth
subalgebras of their respective $C^*$-completions and with a little further
effort one can also show that the same holds for the algebras of smooth
sections of the corresponding bundles of Fr\'echet $*$-algebras (with
regard to their smooth structure as vector bundles).
This line of investigation will be pursued by the authors in the near future.

Another application of our construction of the DFR-bundles is that it provides
nontrivial examples of locally $C^*$-algebras \cite{Ino,Phi,Fra}, namely by
considering the algebras of continuous local sections of our bundles.
The concept of locally $C^*$-algebras is of particular importance for handling 
noncompact spaces and is encountered naturally when dealing with sheaves 
of algebras, so prominent in topology and geometry.
A collection of new results in this direction, related to what has been done
here, can be found in \cite{FP:LC*ANCS}.

We are fully aware of the fact that all these questions are predominantly 
of mathematical nature: the physical interpretation is quite another matter.
But to a certain extent this applies even to the original DFR-model, since
it is not clear how to extend the interpretation of the commutation
relations postulated in~\cite{DFR}, in terms of uncertainty relations,
to other spacetime manifolds, or even to Minkowski space in dimensions
$\neq 4$.
In~addition, it should not be forgotten that, even classically, spacetime
coordinates are \emph{not} observables: this means that the basic axiom of
algebraic quantum field theory according to which observables should be
described by (local) algebras of a certain kind (such as $C^*$-algebras
or von Neumann algebras) does \emph{not at all} imply that in quantum
gravity one should replace classical spacetime coordinate functions
by noncommuting operators.
To us, the basic question seems to be: \emph{How can we formulate spacetime
uncertainty relations, in the sense of obstructions to the possibility of
localizing events with arbitrary precision, in terms of \textbf{observables}?}
That of course stirs up the question:
\emph{How do we actually \textbf{measure} the geometry of spacetime when
quantum effects become strong?}

\section*{Acknowledgments}

Thanks are due to J.C.A.~Barata for suggesting the original problem of
extending the DFR-model to higher dimensions and analyzing its classical
limit using coherent states, which provided the motivation for all the
developments reported here.
Both authors are especially indebted to K.~Fredenhagen for the hospitality
extended to the second author during his stay in Hamburg and for sharing
his insights into the area.
Finally, we are grateful to P.L.~Ribeiro for various valuable suggestions
and remarks.

\begin{appendix}
\renewcommand\theequation{\thesection.\arabic{equation}}
\makeatletter
\@addtoreset{equation}{section}
\makeatother
\setcounter{section}{1}

\section*{Appendix:\ Estimates for the Weyl-Moyal Star Product}
\label{sec:app1}

In this appendix we establish a couple of useful results on the Weyl-Moyal
star product, beginning with an estimate for the Schwartz seminorms of the
product $f \star_\sigma g$ of two functions $f$ and~$g$ in~$\mathcal{B}(V)$
when at least one of them belongs to~$\mathcal{S}(V)$, in terms of the
pertinent seminorms of the factors; as shown in the main text, this
implies a corresponding estimate for the $C^*$-norm on $\mathcal{B}(V)$.
Such estimates can be found in~\cite[Chapters~3~\&~4]{RDQ}, but in our
work we also need some information on how the constants involved in
these estimates depend on the Poisson tensor~$\sigma$, and that part
of the required information is not provided there.
In a second part, we shall discuss the issue of approximate identities
for the Heisenberg-Schwartz algebra (noting that for the Heisenberg-%
Rieffel algebra, this would be a pointless exercise since that already
has a unit, namely the constant function~$1$).

For simplicity, we shall work in coordinates, so we choose a basis
$\, \{ e_1^{},\ldots,e_n^{} \} \,$ of~$V$ and introduce the corresponding
dual basis $\, \{ e^1,\ldots,e^n \} \,$ of~$V^*$, expanding vectors $x$
in~$V$ and covectors $\xi$ in~$V^*$ according to $\, x = x^j e_j^{}$,
$\xi = \xi_j^{} e^j \,$ and the bivector $\sigma$ according to
$\, \sigma(\xi,\eta) = \sigma^{kl} \xi_k^{} \xi_l^{}$; then
\[
 \eta_j^{} (\sigma^\sharp \xi)^j~=~\langle \eta , \sigma^\sharp \xi \rangle~
 =~\sigma(\xi,\eta)~=~\sigma^{kj} \xi_k^{} \eta_j^{}
\]
implies that $\, (\sigma^\sharp \xi)^j = \sigma^{kj} \xi_k^{} \,$.
Moreover, using multiindex notation, we can define the topologies
of~$\mathcal{S}(V)$ and of~$\mathcal{B}(V)$ in terms of the Schwartz
seminorms $s_{p,q}^{}$ (for~$\mathcal{S}(V)$) and $s_{0,q}^{}$ (for~%
$\mathcal{B}(V)$), defined by
\begin{equation} \label{eq:SCHWSN}
 s_{p,q}^{}(f)~=~\sum_{|\alpha| \leqslant p , |\beta| \leqslant q} \,
                \sup_{x \in V} \, |x^\alpha \, \partial_\beta f (x)|~.
\end{equation}

To begin with, we note the following explicit estimate for the $L^1$-norm
of the (inverse) Fourier transform $\check{f}$ of a Schwartz function~$f$
in terms of an appropriate Schwartz seminorm:
\begin{equation} \label{eq:L1EST}
 \| \check{f} \|_1^{}~\leqslant~(2\pi)^n s_{2n,2n}^{}(f)
 \qquad \mbox{for $\, f \in \mathcal{S}(V)$}~.
\end{equation}
\begin{proof}
 \begin{eqnarray*}
  \| \mathcal{F}^{-1} f \|_1^{} \!\!
  &=&\!\! \int_{V^*} d\xi~|(\mathcal{F}^{-1} f)(\xi)| \\[1ex]
  &=&\!\! \int \frac{d\xi_1}{1+\xi_1^2} \ldots \frac{d\xi_n}{1+\xi_n^2}~
          \bigl| \, (1+\xi_1^2) \ldots (1+\xi_n^2)
                    (\mathcal{F}^{-1} f)(\xi) \, \bigr| \\[1ex]
  &\leqslant&\!\! \pi^n \, \sup_{\xi \in V^*} \,
                  \Bigl| \Bigl( \mathcal{F}^{-1}
                  \bigl( (1-\partial_{x^1}^2) \ldots (1-\partial_{x^n}^2) f
                         \bigr) \Bigr) (\xi) \Bigr| \\[1ex]
  &\leqslant&\!\! \frac{1}{2^n} \, \int_V dx~
                  \bigl|  \bigl( (1-\partial_{x^1}^2) \ldots
                                 (1-\partial_{x^n}^2) f \bigr) (x) \bigr|
  \\[1ex]
  &=&\!\! \frac{1}{2^n} \,
          \int \frac{dx^1}{1+(x^1)^2} \ldots \frac{dx^n}{1+(x^n)^2}~
          \Bigl| \, (1+(x^1)^2) \ldots (1+(x^n)^2) \\[-1ex]
  & &\!\! \hphantom{\frac{1}{2^n} \, \int
                    \frac{dx^1}{1+(x^1)^2} \ldots \frac{dx^n}{1+(x^n)^2}~
                    \bigl| \,}
                    \bigl( (1-\partial_{x^1}^2) \ldots (1-\partial_{x^n}^2)
                    f \bigr) (x) \Bigr|  \\
  &\leqslant&\!\! (2\pi)^n s_{2n,2n}^{}(f)~.
 \end{eqnarray*}
\end{proof}

\noindent
Now from equation~(\ref{eq:WMOY7}), we conclude that
\[
 \sup_{x \in V} \, \bigl| (f \star_\sigma g)(x) \bigr|~
 \leqslant~\sup_{x \in V} \, |f(x)| \, \int_{V^*} d\xi~|\check{g}(\xi)|
 \qquad \mbox{for $\, f \in \mathcal{B}(V)$, $g \in \mathcal{S}(V)$}~,
\]
and hence equation~(\ref{eq:L1EST}) gives the following estimate:
\begin{equation} \label{eq:SCHWEST2}
 s_{0,0}^{}(f \star_\sigma g)~\leqslant~(2\pi)^n s_{0,0}^{}(f) s_{2n,2n}^{}(g)
 \qquad \mbox{for $\, f \in \mathcal{B}(V)$, $g \in \mathcal{S}(V)$}~.
\end{equation}
In order to generalize this inequality to higher order Schwartz seminorms,
we need the following facts.
\begin{lem} \label{lem:LRSTP} 
 For $\, f \in \mathcal{B}(V) \,$ and $\, g \in \mathcal{S}(V)$,
 we have the Leibniz rule
 \[
  \frac{\partial}{\partial x^j} \, \bigl( f \star_\sigma g \bigr)~
  =~\frac{\partial f}{\partial x^j} \star_\sigma g \, + \,
    f \star_\sigma \frac{\partial g}{\partial x^j}~,
 \]
 and therefore the higher order Leibniz rule
 \[
  \partial_\alpha \, \bigl( f \star_\sigma g \bigr)~
  =~\sum_{\beta \leqslant \alpha} \binom{\alpha}{\beta} \;
    \partial_\beta f \star_\sigma \partial_{\alpha-\beta} g~.
 \]
\end{lem}
\begin{proof}
 Simply differentiate equation~(\ref{eq:WMOY7}) under the integral sign.
\end{proof}

\begin{lem}
 For $\, f \in \mathcal{B}(V) \,$ and $\, g \in \mathcal{S}(V)$, we have
 \[
  \mathsl{x}^j \bigl( f \star_\sigma g \bigr)~
  =~f \star_\sigma \mathsl{x}^j g \, + \, \nabla_\sigma^j f \star_\sigma g~,
 \]
 and therefore
 \[
  \mathsl{x}^\alpha \bigl( f \star_\sigma g \bigr)~
  =~\sum_{\beta \leqslant \alpha} \binom{\alpha}{\beta} \;
    \nabla_\sigma^{\alpha-\beta} f \star_\sigma \mathsl{x}^\beta g~,
 \]
 where $\nabla_\sigma$ denotes the (pre-)symplectic gradient, defined by
 \[
  \nabla_\sigma^j h~
  =~\frac{i}{2} \, \sigma^{jk} \, \frac{\partial h}{\partial x^k}~.
 \]
 and $\, \nabla_\sigma^\alpha = \prod_{j=1}^n (\nabla_\sigma^j)^{\alpha_j}$.
\end{lem}

\begin{proof}
 For $\, f \in \mathcal{B}(V) \,$ and $\, g \in \mathcal{S}(V)$, we have,
 according to equation~(\ref{eq:WMOY7}),
 \begin{eqnarray*}
  (f \star_\sigma \mathsl{x}^j g)(x) \!\!
  &=&\!\! \int_{V^*} d\xi~
          f(x + {\textstyle \frac{1}{2}} \sigma^\sharp \xi) \,
          (\mathcal{F}^{-1} (\mathsl{x}^j g))(\xi) \;
          e^{i \langle \xi, x \rangle} \\
  &=&\!\! i \, \int_{V^*} d\xi~
          f(x + {\textstyle \frac{1}{2}} \sigma^\sharp \xi) \,
          \frac{\partial \check{g}}{\partial \xi_j^{}}(\xi) \;
          e^{i \langle \xi, x \rangle} \\
  &=&\!\! - \, i \, \int_{V^*} d\xi
          \biggl( \Bigl( \frac{\partial}{\partial \xi_j^{}}
                         f(x + {\textstyle \frac{1}{2}} \sigma^\sharp \xi)
                         \Bigr) \, \check{g}(\xi) \; e^{i \langle \xi, x \rangle}
  \\[-1ex]
  & & \mbox{} \hspace*{5em} + \,
          f(x + {\textstyle \frac{1}{2}} \sigma^\sharp \xi) \,
          \check{g}(\xi) \; \frac{\partial}{\partial \xi_j^{}} \,
          e^{i \langle \xi, x \rangle} \biggr) \\
  &=&\!\! \frac{i}{2} \, \sigma^{kj}
          \int_{V^*} d\xi~\frac{\partial f}{\partial x^k}
          (x + {\textstyle \frac{1}{2}} \sigma^\sharp \xi) \,
          \check{g}(\xi) \; e^{i \langle \xi, x \rangle} \, + \,
          x^j (f \star_\sigma g)(x) \\[1ex]
  &=&\!\! \mbox{} - \, (\nabla_\sigma^j f \star_\sigma g)(x) \, + \,
          x^j (f \star_\sigma g)(x)~.
 \end{eqnarray*}
\end{proof}

\noindent
Combining these two lemmas gives the formula
\begin{equation} \label{eq:PRSTP} 
 \begin{array}{c}
  {\displaystyle
   \mathsl{x}^\alpha \partial_\beta^{} (f \star_\sigma g)~
   =~\sum_{\gamma \leqslant \alpha , \delta \leqslant \beta} \,
     \binom{\alpha}{\gamma} \binom{\beta}{\delta} \;
     \bigl( \nabla_\sigma^{\alpha-\gamma} \partial_{\beta-\delta}^{} f \,
            \star_\sigma \mathsl{x}^\gamma \partial_\delta^{} g \bigr)}
  \\[3ex]
  \mbox{for $\, f \in \mathcal{B}(V)$, $g \in \mathcal{S}(V)$}~.
 \end{array}
\end{equation}
Taking the sup norm (which is just $s_{0,0}^{}$) and applying the definition
of the seminorms $s_{p,q}^{}$ together with the estimate~(\ref{eq:SCHWEST2})
established above, we arrive at the following
\begin{prp} \label{prp:SCHWEST} 
 For any two natural numbers $p$, $q$, there exists a polynomial $P_{p,q}$
 of degree~$p$ in the matrix elements of $\sigma$, with coefficients that
 depend only on $n$, $p$ and~$q$, such that the following estimate holds:
\begin{equation}
 \begin{array}{c}
  s_{p,q}^{}(f \star_\sigma g)~
  \leqslant~|P_{p,q}^{}(\sigma)| \, s_{0,p+q}^{}(f) \, s_{p+2n,q+2n}^{}(g)
  \\[1ex]
  \mbox{for $\, f \in \mathcal{B}(V)$, $g \in \mathcal{S}(V)$}~.
 \end{array}
\end{equation}
\end{prp}

With these formulas and estimates at our disposal, we can address the
issue of constructing approximate identities for the Heisenberg-Schwartz
algebra $\mathcal{S}_\sigma$. The fact that this is a $*$-subalgebra (and
even a $*$-ideal) of the Heisenberg-Rieffel algebra~$\mathcal{B}_\sigma$
which does have a unit, namely the constant function~$1$, indicates that
we should look for sequences $(\chi_k)_{k \in \mathbb{N}}$ of Schwartz
functions $\, \chi_k \in \mathcal{S}_\sigma \,$ which converge to~$1$
in some appropriate sense: without loss of generality, we may assume
these functions to be real-valued and to satisfy $\, 0 \leqslant \chi_k
\leqslant 1$.
Thus we expect that $\, \chi_k \to 1 \,$ and $\, \partial_\alpha \chi_k
\to 0 \,$ for $\, \alpha \neq 0 \,$ (or equivalently, $\partial_\alpha
(1 - \chi_k) \to 0 \,$ for all $\alpha$) as $\, k \to \infty$, but
this convergence can at best hold uniformly on compact subsets of~$V$.%
\footnote{Typically, we may even take $(\chi_k)_{k \in \mathbb{N}}$ to
be a sequence of test functions $\, \chi_k \in \mathcal{D}(V) \,$ that
is monotonically increasing and converges to~$1$ in~$\mathcal{E}(V)$.
Note, however, that this sequence does not converge to~$1$ in the
space $\mathcal{S}(V)$ and not even in the space $\mathcal{B}(V)$,
since the function~$1$ does not go to~$0$ at infinity: convergence
is only uniform on compact subsets but not on the entire space.}
Still, it turns out that any such sequence yields an approximate
identity for the Heisenberg-Schwartz algebra~-- provided we also
require the partial derivatives $\partial_\alpha (1 - \chi_k)$ to
be uniformly bounded in~$k$, for all $\alpha$:
\begin{prp} \label{prp:APRIHS} 
 Let $(\chi_k)_{k \in \mathbb{N}}$ be a sequence of Schwartz functions
 $\, \chi_k \in \mathcal{S}(V) \,$ satisfying \linebreak $0 \leqslant
 \chi_k \leqslant 1 \,$ which is bounded in the Fr\'echet space
 $\mathcal{B}(V)$ and converges to~$1$ in the Fr\'echet space
 $\mathcal{E}(V)$, that is, in the topology of uniform convergence
 of all derivatives on compact subsets.
 Then $(\chi_k)_{k \in \mathbb{N}}$ is an approximate identity in the
 Heisenberg-Schwartz algebra $\mathcal{S}_\sigma$, i.e., for any
 $\, f \in \mathcal{S}_\sigma$, we have that $\, \chi_k \star_\sigma f
 \to f \,$ in~$\mathcal{S}_\sigma$, as $\, k \to \infty$.
\end{prp}
\begin{proof}~
 Fixing $\, f \in \mathcal{S}_\sigma \,$ and $\, p,q \in \mathbb{N}$, we have
 the following estimate,
 \[
  s_{p,q}^{}(\chi_k \star_\sigma f - f)~
  \leqslant~C_0 \, \max_{\genfrac{}{}{0pt}{}{|\alpha|,|\gamma| \leqslant p}%
                                            {|\beta|,|\delta|  \leqslant q}} \;
            \sup_{x \in V} \,
            \bigl| \bigl( \nabla_\sigma^\alpha \partial_\beta^{} (1-\chi_k)
                          \, \star_\sigma \,
                          \mathsl{x}^\gamma \partial_\delta^{} f \bigr) (x) \,
                   \bigr|~,
 \]
 where
 \[
  C_0~=~\sum_{|\alpha| \leqslant p , |\beta| \leqslant q}
            \sum_{\gamma \leqslant \alpha , \delta \leqslant \beta\vphantom{|}} \,
            \binom{\alpha}{\gamma} \binom{\beta}{\delta}~,
 \]
 which follows directly from equation~(\ref{eq:PRSTP}) after some relabeling.
 Now given $\, \epsilon > 0$, we shall split this sup norm into two parts.
 First, we use that the functions $\chi_k$, and hence also the functions
 $\nabla_\sigma^\alpha \partial_\beta^{} (1-\chi_k)$, form a bounded subset
 of~$\mathcal{B}_\sigma$, while the $\mathsl{x}^\gamma \partial_\delta^{} f$
 are fixed functions in~$\mathcal{S}_\sigma$, to conclude that there exists a
 compact subset $K$ of~$V$ such that, for all $|\alpha|,|\gamma| \leqslant p$
 and $|\beta|,|\delta| \leqslant q$,
 \[
  \sup_{x \notin K} \,
  \bigl| \bigl( \nabla_\sigma^\alpha \partial_\beta^{} (1-\chi_k) \, \star_\sigma \,
                \mathsl{x}^\gamma \partial_\delta^{} f \bigr) (x) \, \bigr|~
  <~\frac{\epsilon}{C_0}~.
 \]
 Indeed, we may apply equations~(\ref{eq:SCHWEST2}) and~(\ref{eq:PRSTP})
 to show that the Schwartz functions \linebreak $(1+|\mathsl{x}|^2)
 \bigl( \nabla_\sigma^\alpha \partial_\beta^{} (1-\chi_k) \star_\sigma
 \mathsl{x}^\gamma \partial_\delta^{} f) \,$ on~$V$ are uniformly
 bounded in~$k$ (as well as in all other parameters), so the
 Schwartz functions $\, \nabla_\sigma^\alpha \partial_\beta^{}
 (1-\chi_k) \star_\sigma \mathsl{x}^\gamma \partial_\delta^{} f \,$
 on~$V$ vanish at infinity uniformly in~$k$ (as well as in all
 other parameters). Next, we set
 \[
  C_1~=~\max_{|\alpha| \leqslant p , |\beta| \leqslant q} \,
        s_{0,0}^{} \bigl( \nabla_\sigma^\alpha \partial_\beta^{} (1-\chi_k) \bigr)
  ~~,~~
  C_2~=~\max_{|\gamma| \leqslant p , |\delta| \leqslant q} \,
        \| \mathcal{F}^{-1} (\mathsl{x}^\gamma \partial_\delta^{} f) \|_1^{}        
 \]
 and introduce a compact subset $K^*$ of~$V^*$ such that, for all
 $|\gamma| \leqslant p$ and $|\delta| \leqslant q$,
 \[
  \int_{V^* \setminus K^*} d\xi~
  |\mathcal{F}^{-1} (\mathsl{x}^\gamma \partial_\delta^{} f)(\xi)|~
  <~\frac{\epsilon}{2 C_0 C_1}~.
 \]
 Now let $\, L = K + \frac{1}{2} \sigma^\sharp K^*$, which is again a compact
 subset of~$V$, and finally use the uniform convergence of the functions
 $1 - \chi_k$ and their derivatives on~$L$ to infer that there exists
 $\, k_0 \in \mathbb{N} \,$ such that, for $\, k \geqslant k_0 \,$ and
 all $|\alpha| \leqslant p$ and $|\beta| \leqslant q$,
 \[
  \sup_{y \in L} \, \bigl| \bigl( \nabla_\sigma^\alpha \partial_\beta^{}
                                (1-\chi_k) \bigr) (y) \bigr|~
  <~\frac{\epsilon}{2 C_0 C_2}~.
 \]
 Then it follows from equation~(\ref{eq:WMOY7}) that, for $\, k \geqslant
 k_0 \,$ and all $|\alpha|,|\gamma| \leqslant p$ and $|\beta|,|\delta|
 \leqslant q$,
 \begin{eqnarray*}
 \lefteqn{
  \sup_{x \in K} \,
  \bigl| \bigl( \nabla_\sigma^\alpha \partial_\beta^{} (1-\chi_k) \, \star_\sigma \,
                \mathsl{x}^\gamma \partial_\delta^{} f \bigr) (x) \, \bigr|}
  \hspace{1cm} \\[1ex]
  &\leqslant&\!\! \left| \, \int_{V^* \setminus K^*} d\xi~
                  (\nabla_\sigma^\alpha \partial_\beta^{} (1-\chi_k))
                  (x + {\textstyle \frac{1}{2}} \sigma^\sharp \xi) \,
                  (\mathcal{F}^{-1} (\mathsl{x}^\gamma \partial_\delta^{} f))
                  (\xi) \; e^{i \langle \xi, x \rangle} \, \right| \\
  &         &\!\! + \, \left| \, \int_{K^*} d\xi~
                  (\nabla_\sigma^\alpha \partial_\beta^{} (1-\chi_k))
                  (x + {\textstyle \frac{1}{2}} \sigma^\sharp \xi) \,
                  (\mathcal{F}^{-1} (\mathsl{x}^\gamma \partial_\delta^{} f))
                  (\xi) \; e^{i \langle \xi, x \rangle} \, \right| \\[1ex]
  &\leqslant&\!\! s_{0,0}^{} \bigl( \nabla_\sigma^\alpha \partial_\beta^{}
                                   (1-\chi_k) \bigr) \;
                  \int_{V^* \setminus K^*} d\xi~|(\mathcal{F}^{-1}
                  (\mathsl{x}^\gamma \partial_\delta^{} f))(\xi)| \\
  &         &\!\! + \, \sup_{y \in L} \,
                  \bigl| \bigl( \nabla_\sigma^\alpha \partial_\beta^{}
                                (1-\chi_k) \bigr) (y) \bigr| \;
                  \int_{V^*} d\xi~|\mathcal{F}^{-1}
                  (\mathsl{x}^\gamma \partial_\delta^{} f)(\xi)| \\[2ex]
  &    <    &\!\! \frac{\epsilon}{C_0}~.
 \end{eqnarray*}
\end{proof}

\noindent
It may be worthwhile to emphasize that this construction provides
an entire class of approximate identities for the Heisenberg-Schwartz
algebra but no bounded ones: the $\chi_k$ are uniformly bounded in~$k$
only in $\mathcal{B}_\sigma$ but not in~$\mathcal{S}_\sigma$.
This is unavoidable since it is in fact not difficult to prove
that the Heisenberg-Schwartz algebra does not admit any bounded
approximate identities, but we shall not pursue the matter any
further since we do not need this fact in the present paper.

\end{appendix}

\end{document}